\documentclass[a4paper, english, 11pt]{amsart}
\usepackage{babel}

\usepackage[colorinlistoftodos]{todonotes}

\usepackage[utf8x]{inputenc}

\usepackage{amsmath}
\usepackage{amsfonts}
\usepackage{amsthm}
\usepackage{amssymb}

\usepackage{bbm}

\usepackage[shortlabels]{enumitem}

\usepackage{calrsfs}

\usepackage{booktabs}

\usepackage[all]{xy}

\usepackage{extarrows}

\makeatletter
\newcommand{\LeftEqNo}{\let\veqno\@@leqno}
\makeatother

\newtheoremstyle{break}
  {}
  {}
  {\itshape}
  {}
  {\bfseries}
  {.}
  {\newline}
  {}

\swapnumbers

\theoremstyle{break}
\newtheorem{The}{Theorem}%
[section]

\theoremstyle{plain}
\newtheorem{Pro}[The]{Proposition}
\newtheorem{Cor}[The]{Corollary}
\newtheorem{Lem}[The]{Lemma}

\theoremstyle{definition}

\newcommand{\aut}[1]{\operatorname{Aut}\left({#1}\right)}

\newcommand{\frattini}{\Phi}
\newcommand{\Z}{\operatorname{Z}}

\usepackage[pdfauthor={Giuliano Bianco},%
pdftitle={On connected quandles of prime power order},%
pagebackref=true,%
pdftex]{hyperref}

\usepackage{tikz}
\usepackage{subfig} 
\usepackage{caption}

\def\setof#1#2{\{#1\,\mid\,#2\}}

\newcommand{\disp}{displacement}
\newcommand{\leftm}{left multiplication}
\newcommand{\lmlt}[1]{\mathrm{LMlt}(#1)}
\newcommand{\dis}{\mathrm{Dis}}
\def\c#1{\mathrm{con}_{#1}}
\usepackage{MnSymbol}
\def\Q{\mathcal{Q}_{\mathrm{Hom}}}
\def\aff#1{\mathrm{Aff}#1}

\newtheorem{example}[The]{Example}

\newtheorem{remark}[The]{Remark}

\newcommand{\Lm}[1]{L_{#1}}

\def\Z{Z(\dis(Q))}
\def\ldiv{\backslash}
\def\inv{^{-1}}

\newcommand{\BlocS}[2]{\dis(#1)_{[#2]}}

\begin{document}

\title{On connected quandles of prime power order}

\author{Giuliano Bianco, Marco Bonatto}
\date{v\today}



\begin{abstract}
We develop some general ideas to study connected quandles of prime power size and we classify non-affine connected quandles of size $p^3$ for $p>3$, using a combination of group theoretical and universal algebraic tools. As a byproduct we obtain a classification of Bruck loops of the same size and that they are in one-to-one correspondence with commutative automorphic loops.
\end{abstract}

\maketitle

\section*{Introduction}

Quandles are binary algebras arising in the context of knot theory \cite{Joyce1, Mat}, Hopf algebras \cite{AG2003} and the set theoretic solution to the Yang-Baxter equation \cite{EGS,ESS}.

Some families of finite connected quandles have been classified: connected quandles of size $p$ \cite{EGS}, of size $p^2$ \cite{Grana_p2}, affine quandles of size up to $p^4$ \cite{Hou} and quandles with doubly transitive left multiplication group \cite{Vendra}. Moreover a list of connected quandles of size less or equal to $47$ can be found in the RIG library on GAP \cite{RIG}. The classification strategies used in these papers make use either of the strong interplay between connected quandles and transitive groups, in particular using the minimal {\it coset representation} over the \disp\, group \cite{hsv}  (i.e. a representation of the quandle parametrized by a pair $(G,f)$ where $G$ is a group and $f\in \aut{G}$, see Example \ref{coset quandles}), or the representation of affine quandles as modules over the Laurent series with integer coefficients.
We will integrate this two approaches with a new viewpoint coming from universal algebra, which reveals another link between the structure of a quandle and the structure of its \disp\, group, namely the existence of a Galois connection between the congruence lattice of a quandle and certain subgroups of the \disp\, group (see Section \ref{Sec:Galois}). We are also going to use the commutator theory for quandles developed in \cite{CP}, in which commutator theory for general algebras started by J.D. Smith, Freese and McKenzie \cite{comm} is adapted to the framework of quandles (see Section \ref{Sec:commutator}). In \cite{CP} the definition of nilpotent quandles is provided wich, as it is customary with groups, can be described in terms of the existence of a lower central series of congruences
\begin{displaymath}
1_Q=\gamma_0\leq \gamma_1(Q)\leq\ldots \leq \gamma_n(Q)=0_Q
\end{displaymath} 
where $0_Q$ and $1_Q$ are respectively the smallest and the biggest congruence of $Q$ and $\gamma_{i+1}(Q)=[\gamma_{i}(Q),1_Q]$, where $[\cdot,\cdot]$ denotes the commutator of congruences defined in \cite{comm}. In such case $Q$ is said to be {\it nilpotent of length $n$}. According to this characterization, nilpotence of a quandle is reflected by nilpotence of its displacement group, see \cite[Section 6]{CP}. With a nice analogy to what happens with groups, it is proved that prime power size connected quandles are nilpotent as a byproduct of Corollary 5.2 of \cite{GiuThe}. Moreover, according to Proposition 6.5 of \cite{CP}, nilpotent latin quandles are direct product of quandles of prime power size, so one can focus on the prime power size case. 

In Section \ref{preliminary} we collect some basic facts about quandles, we recall the contents of \cite{CP} useful in the present work and we prove a characterization of the congruence $\gamma_1(Q)$ (Proposition \ref{gamma_1}). In Section \ref{derived} we focus on the role of some particular subgroups of the \disp\, group as the first and second element of the lower central series and the Frattini subgroup and how they are related to the structure of the quandle. In Section \ref{central} we investigate central congruences and nilpotent quandles of length $2$. In Section \ref{p-racks} we apply our tools to re-obtain a characterization of connected quandles of size $p$ and $p^2$ in the contest of the techniques here developed.

In the Section \ref{cube order} we obtain a classification of non-affine connected quandles of size $p^3$ with $p> 3$, using the following strategy: first we show that they are nilpotent of length $2$ and then we find a bound on the size of the \disp\, group in order to exploit the minimal coset representation. Then we give necessary conditions for pairs $(G,f)$ to give the desired minimal representation and we show that isomorphism classes of quandles over a fixed group $G$ correspond to conjugacy classes of its automorphisms. Finally, taking advantage of the classification of groups of size up to $p^n$ with $n\leq 4$ and the characterization of their automorphisms available in \cite{tedesco} we get a list of representatives of isomorphism classes collected in tables \ref{Tab2} and \ref{Tab1}. For all the (technical) computations of the relevant conjugacy classes we refer to the Appendix \ref{appendix}.

We finish this section with an application to Loop theory. Indeed, as a byproduct, we obtain also a classification of Bruck Loops of order $p^3$, since they correspond to involutory latin quandles of the same size. Using the enumeration of commutative automorphic loops of size $p^3$ we have that they are in one-to-one correspondence with Bruck Loops of the same size (providing an answer to Problem 8.1 in \cite{Petr2} for every prime $p$ and $k=3$).

In Section \ref{bound for latin} we compute a bound for the size of the \disp \, group of latin quandles of arbitrary prime power size as a first step of the same strategy we used in the present work applied to a more general case.

\section{Preliminary results}\label{preliminary}

A \textit{rack} is a pair $(Q,\ast)$ where $Q$ is a set and $\ast$ is a binary operation and
\begin{itemize}
    \item[(i)] the left multiplication $\Lm{a} : Q \to Q,$ mapping $ b \mapsto a \ast b$ is a bijection for every $a\in Q$;
    \item[(ii)] $a\ast (b\ast c)=(a\ast b)\ast (a\ast c)$ for every $a,b,c\in Q$.
\end{itemize}
A \textit{quandle} is an idempotent rack, i.e. $a* a=a$ for every $a\in Q$. Equivalently a rack is a left distributive left quasigroup, i.e. a binary algebra $(Q,*,\ldiv)$ such that $a*(a\ldiv b)=a\ldiv (a*b)=b$ for every $a\in Q$ and the left multiplication mappings are automorphisms.

\begin{example}
Let $G$ be a group and $H\subseteq G$ be closed under conjugation. Then $(H,*)$ with $h*g=h g h \inv$ is a quandle.
\end{example}

\begin{example}\label{coset quandles}
Let $G$ be a group, $g\in \operatorname{Aut}{G}$ and $H\leq Fix(f)=\setof{g\in G}{f(g)=g}$, the set $G/H$ endowed with the operation:
\begin{displaymath}
aH\ast bH=af(a^{-1}b)H,
\end{displaymath}
is a quandle and it is denoted by $\Q(G,H,f)$ and called \textit{coset quandle}. If $H=1$, then $Q$ is called \textit{principal} and denoted by $Q=\Q(G,f)$. If $G$ is Abelian, $Q$ is called \textit{affine} and denoted by $\aff(G,f)$ and the quandle operation becomes
\begin{displaymath}
a\ast b=(1-f)(a)+f(b)
\end{displaymath}
for every $a,b\in Q$.
\end{example}

$\Lm{a}$ and $R_a$ will be the left and right multiplications, respectively, for an element $a$ and the group generated by the left multiplications is denoted by $\lmlt{Q}$ and called 
\emph{\leftm} group. The group generated by $\{\Lm{a}\Lm{b}^{-1}, \, a,b\in Q\}$ is denoted by $\dis(Q)$ and called \emph{\disp} group (by some authors these two groups are called the Inner Automorphisms group and the Transvection group respectively). If $\aut{Q}$ is transitive $Q$ is called \textit{homogeneous}, if $\lmlt{Q}$ is transitive $Q$ is called \textit{connected}.
The construction defined in Example \ref{coset quandles} characterizes homogeneous quandles \cite[Theorem 7.1]{Joyce1}.  
In particular the \disp\, group provides a minimal coset representation for connected quandles as 
\begin{equation}\label{min rep}
Q \cong \Q(\dis(Q),\dis(Q)_a,\widehat{\Lm{a}}),
\end{equation}
where $\dis(Q)_a$ is the point-wise stabilizer of $a$ in $\dis(Q)$ and $\widehat{\Lm{a}}$ is the inner automorphism of $\Lm{a}$ \cite[Theorem 4.1]{hsv}. 

An equivalence relation on the elements of a rack is a \textit{congruence} if it is compatible with the algebraic structure of the rack. More precisely, if $Q$ is a rack, an equivalence relation $\alpha$ is a congruence of $Q$ whenever $a \, \alpha \, b$ and $c\, \alpha \, d$ implies $ \left(a\ast c \right) \, \alpha \, \left( b\ast d\right)$ and $ \left(a\ldiv c \right) \, \alpha \, \left( b\ldiv d\right)$. We denote by $[a]_{\alpha}$ the class of $a$ (we usually drop the subscript $\alpha$ whenever there is no risk of confusion) and by $Q/\alpha$ the quotient set. Congruences are in one-to-one correspondence with kernel of homomorphisms, where the kernel of a map $f:Q\to R$ is the equivalence relation given by $ker(f)=\setof{(a,b)\in Q\times Q}{f(a)=f(b)}$. The congruences of a quandle $Q$ form a lattice with largest element $1_Q =Q\times Q$ and smallest element $0_Q=\{(a,a);\;a\in Q\}$. This lattice will be denoted by $Con(Q)$. The congruence lattice of $Q/\alpha$ is given by $\setof{\beta/\alpha}{\alpha\leq \beta\in Con(Q)}$.

Let $Q$ be a rack and $\alpha$ a congruence of $Q$. The mapping

\begin{displaymath}
\pi_{\alpha}:\lmlt {Q}\longrightarrow \lmlt{Q/\alpha},\quad \Lm{a_1}\ldots \Lm{a_n} \mapsto \Lm{[a_1]}\ldots \Lm{[a_n]}
\end{displaymath}
is well defined and is a surjective homomorphism of groups (see \cite[Lemma 1.8]{AG2003}). The restriction of $\pi_\alpha$ to $\dis (Q)$ gives a surjective homomorphism $\dis (Q)\to\dis(Q/\alpha)$. Its kernel will be denoted by $\dis^\alpha$ and it coincides with the set 
\begin{equation}\label{kernel}
\dis^\alpha=\setof{h\in \dis(Q)}{h([a])=[a], \text{ for every }a\in Q}.
\end{equation}

Another subgroup of the \disp\; group relevant to our discussion is the following: 
\begin{equation}\label{relative dis}
\dis_\alpha=\langle \setof{L_aL_b^{-1}}{ a\, \alpha\, b}\rangle.
\end{equation}
which can be somehow be described as the \disp\; group relative to $\alpha$.

The set-wise block stabilizer
\begin{equation}\label{block_stab}
\dis(Q)_{[a]_\alpha}=\setof{ h\in \dis(Q)} { h([a]_\alpha)=[a]_\alpha}
\end{equation}
provides the following useful criterion for connectedness where $\lmlt{Q}_{[a]_\alpha}$ is defined analogically to 
\eqref{block_stab} (a similar statement is implicit in \cite{AG2003} at page 196): 
\begin{Pro}\label{connected ext}
Let $Q$ be a rack (resp. a quandle), $\alpha$ be its congruence and $a\in Q$. Then $Q$ is connected if and only if $Q/\alpha$ is connected and $\lmlt{Q}_{[a]_\alpha}$ (resp. $\dis(Q)_{[a]_\alpha}$) is transitive on $[a]_{\alpha}$.
\end{Pro}
\begin{proof}
If $Q$ is a connected rack then $Q/\alpha$ is connected \cite[page 184]{AG2003} and for every $b\in [a]$ there $h\in \lmlt{Q}$ such that $b=h(a)$. The blocks of $\alpha$ are blocks with respect to the action of $\lmlt{Q}$, then whenever $c\, \alpha\, a$ then $h(c)\,\alpha\, h(a)=b\, \alpha \, a$. Therefore $h\in \lmlt{Q}_{[a]}$ and the block stabilizer is indeed transitive on the block $[a]$. On the other hand if $Q/\alpha$ is connected then for every $b\in Q$ there exists $\pi_\alpha(h)\in \lmlt{Q/\alpha}$ such that $[b]=\pi_\alpha(h)([a])$. Then $	h(a)\, \alpha	\, b$ and $a\, \alpha\, h^{-1}\left( b \right) $. Since the block stabilizer is transitive on $[a]$ we have that $h^{-1}\left( b \right)=g(a)$ for some $g\in \lmlt{Q}_{[a]}$ which implies $b=hg\left( a \right)$. Hence $Q$ is connected. The statement for quandles follows from the fact that in their case $\dis(Q)$ and $\lmlt{Q}$ have the same orbits.
\end{proof}

The stabilizer of $a\in Q$ in $\dis(Q)$ is a subgroup of the block stabilizer \eqref{block_stab}. Indeed, $h\in \dis(Q)_a$, then $h(b) \, \alpha \, h(a)=a$ whenever $b\, \alpha \, a$. So 
\begin{equation}\label{stabiliser in block stabiliser}
\dis(Q)_a\leq \BlocS{Q}{a}. 
\end{equation}
Moreover we have:
\begin{equation}\label{i sottogruppi}
\dis_\alpha\leq \dis^\alpha\leq \dis(Q)_{[a]}.
\end{equation}

Note that $L(Q)=\setof{\Lm{a}}{a\in Q}$ is a subset of $\lmlt{Q}$ closed under conjugation, this follows from the fact that $\widehat{\Lm{a}}\left(\Lm{b}\right)=\Lm{a} \Lm{b} \Lm{a}^{-1}=L_{a\ast b}$, and as such it has a natural quandle structure. The mapping $L_Q:Q\to L(Q)$ is an epimorphism of quandles with kernel denoted by $\lambda_Q$, i.e. $a\,\lambda_Q\, b$ if and only if $\Lm{a}= \Lm{b}$. A quandle is called {\it faithful} if $\lambda_Q=0_Q$ i.e. $a = b$ if and only if $\Lm{a}= \Lm{b}$, in this case $\dis(Q)_a=Fix(\widehat{\Lm{a}})$ for every $a\in Q$. A quandle is called {\it latin} if right-multiplications $R_a:b\mapsto b\ast a$ are bijections.  Note that latin quandles are connected and faithful.

\begin{Lem}\label{facts}
The following easy and well know facts will be freely used in the rest of the paper:
\begin{itemize}
\item[(i)] If $Q$ is a finite connected quandle then $|Q|=|Q/\alpha||[a]|$ for every congruence $\alpha$ and every $a\in Q$ \cite[Proposition 2.5]{CP}.
\item[(ii)] A connected quandle $Q$ is principal if and only if $\dis(Q)_a=1$ for every $a\in Q$ and a finite principal quandle $\Q(G,f)$ is latin if and only if $Fix(f)=1$, \cite[Section 2.1]{Principal}. Finite affine connected quandles are latin \cite[Section 2]{hsv}.
\end{itemize}
\end{Lem}

For further definitions and basic results on quandles we refer the reader to \cite{AG2003, Joyce1, hsv}. 

\subsection*{Congruences and subgroups of the \disp\, group}\label{Sec:Galois}
There is a close relationship between the congruence lattice of a rack $Q$ and the lattice of normal subgroups of its left multiplication group which are contained in the displacement subgroup (we denote this last lattice of subgroups by $Norm(Q)$). It is a well known fact that $\lmlt{Q}=\dis(Q)\langle L_a\rangle$ for every $a\in Q$ and so it is easy to see that the elements of $Norm(Q)$ are the normal subgroups of $\dis(Q)$ invariant under $\widehat{\Lm{a}}$.

For every $N \in Norm(Q)$ we can build the congruence 
$$\c N=\setof{(a,b)}{\Lm{a}\Lm{b}^{-1}\in N }\in Con(Q).$$
On the other hand to every $\alpha\in Con(Q)$ we can associate the subgroup $\dis_\alpha\in Norm(Q)$ defined in \eqref{relative dis}. The pair of mappings $(\dis,\c{})$ is a monotone Galois connection (see \cite[Section 3.3]{CP} for more details). 

Another congruence relevant to the present work is the orbit decomposition of $Q$ with respect to the action of $N\in Norm(Q)$ as
\begin{displaymath}
a\,\mathcal{O}_N b\,\text{ if and only if there exists }\, n \in N \, \text{such that }\, n\left(a\right)=b.
\end{displaymath}
Such equivalence relation is indeed a congruence and it will be denoted by $\mathcal{O}_N$ \cite[Section 2.3]{CP}.

The interplay between the two lattices $Con(Q)$ and $Norm(Q)$ will be useful to extract information about the structure of the displacement group, given information about the congruence lattice and vice versa. 

For the reader's convenience let's state some simple facts about $Con(Q)$ and $Norm(Q)$ that will be freely used in developing the arguments. Let $\alpha\in Con(Q)$ and $N\in Norm(Q)$, then:

\begin{eqnarray}
& &\mathcal{O}_N \leq \c{N}, \\
& & \dis_{\c N} \leq	N\leq \dis^{\mathcal{O}_N}\leq \dis^{\c N},\label{N in O_N}\\
&& \mathcal{O}_{\dis^\alpha}\leq \alpha .\label{orbits of kernel}
\end{eqnarray}


\subsection*{Commutator theory for racks}\label{Sec:commutator}
Congruences of groups are in one to one correspondence with normal subgroups, and one can define the notion of commutator between two (normal) subgroups which can be understood as a binary operation on the lattice of (normal) subgroups. In analogy, one can define the commutator of two congruences (denoted by brackets $[\cdot,\cdot]$) and consequently the notion of abelianness and centrality of congruences for general algebras (in the sense of universal algebra) \cite{comm}. We will state such definitions directly in the setting of quandles.
%

Let $Q$ be a quandle. A congruence $\alpha$ is abelian if $[\alpha,\alpha]=0_Q$ and $Q$ is called \emph{abelian} if the congruence $1_Q$ is abelian, i.e. $[1_Q,1_Q]=0_Q$.

A congruence $\alpha$ is \emph{central} if $[\alpha,1_Q]=0_Q$ and the \emph{center} of $Q$, denoted by $\zeta_Q$, is the largest central congruence of $Q$. Equivalently, a quandle is abelian if $\zeta_Q=1_Q$. The congruence $[1_Q,1_Q]$ is the smallest congruence $\delta$ such that $Q/\delta$ is abelian. 

Once you have the notion of a commutator you can define the lower central series in analogy with groups as 
\begin{displaymath}
    \gamma_{0}(Q)=1_Q,\qquad \gamma_{i+1}(Q)=[\gamma_{i}(Q),1_Q].
\end{displaymath}
A quandle $Q$ is called \emph{nilpotent of length $n$}, if $\gamma_{n}(Q)=0_Q$ for some $n$.

In \cite{CP} commutator theory for racks in the sense of \cite{comm} has been studied. The centrality of congruences of racks is completely determined by the properties of $\dis(Q)$ and its subgroups.
\begin{The}[{\cite[Theorem 1.1]{CP}}]
\label{central cong}Let $Q$ be a rack and $\alpha$ be a congruence. The following are equivalent:
\begin{enumerate}
\item[(i)] $\alpha$ is central;
\item[(ii)] $\dis_\alpha $ is central in $\dis(Q)$ and $\dis(Q)_a=\dis(Q)_b$ whenever $a\, \alpha\, b$.
\end{enumerate} 
In particular if $Q$ is connected, $Q$ is abelian if and only if $\dis(Q)$ is abelian.
\end{The}
A characterization of the congruence $\gamma_1(Q)$ is the following:

\begin{Pro}\label{gamma_1}
Let $Q$ be a connected quandle. Then $\gamma_1(Q)=\mathcal{O}_{\gamma_1(\dis(Q))}$.
\end{Pro}

\begin{proof}
Let $G=\dis(Q)$ and $\alpha=\mathcal{O}_{\gamma_1(G)}$. We have $\gamma_1(G)\leq \dis^\alpha$ \eqref{N in O_N}, hence $G/\dis^\alpha$ is abelian and regular since $Q/\alpha$ is connected. Then applying Theorem \ref{central cong} to the congruence $1_{Q/\alpha}$ it can be inferred that $Q/\alpha$ is abelian. Let $\beta\in Con(Q)$ such that $Q/\beta$ is abelian. Then $\dis(Q/\beta)$ is abelian and then $\gamma_1(G)\leq \dis^\beta$ and so the orbits of $\gamma_1(G)$ are contained in the orbits of $\dis^\beta $ i.e. $\alpha\leq \mathcal{O}_{\dis^\beta}\leq \beta$ where the last inclusion follows by \eqref{orbits of kernel}. Hence $\alpha=\gamma_1(Q)$. 
\end{proof}
 A description of the center of a rack can also be obtained via the Galois connection as $\zeta_Q=\c{Z(\dis(Q))}\bigcap \sigma_Q$, where $a\, \sigma_Q\, b $ whenever $\dis(Q)_a=\dis(Q)_b$ (Proposition 5.9 \cite{CP}). Note that if $Q$ is a principal quandle then $\sigma_Q=1$, if instead $Q$ is faithful then $\zeta_Q\leq \sigma_Q$ \cite[Section 5.2]{CP}. So in both cases we have 
\begin{equation}\label{zeta_Q_Con_zeta_Q}.
 \zeta_Q=\c{Z(\dis(Q))}.
\end{equation}



\section{Derived Subgroup Lemma}\label{derived}

Let $G$ be a group and $f$ one of its automorphisms. We now prove a useful group theoretical Lemma involving the mapping 
\[\partial f:G\longrightarrow G,\quad g\mapsto g f(g)^{-1}\]
and the subgroup $[G,f]=\langle Im(\partial f)\rangle=\langle gf(g)^{-1}, \, g\in G\rangle$. We are interested in this subgroup because for a coset quandle $Q=\Q(G,H,f)$, the action of the a generators of the \disp\; group of $Q$ is given by 
$$L_{gH}L_H^{-1}(aH)=\partial f(g) aH.$$
So the orbit of $aH$ under $\dis(Q)$ is equal to the orbit of the left action of $[G,f]$ on it and $\dis(Q)\cong [G,f]/ Core_G(H)$ \cite[Lemma 3.11]{GiuThe}. Accordingly we have the following:

\begin{Lem}\label{commutator_with_trans_is_trans}
If a quandle $Q$ is connected then $[\dis(Q),\widehat{\Lm{a}}]\cong\dis(Q)$ for every $a\in Q$. 
\end{Lem}
\begin{proof}
It follows from the previous observation applied to the coset representation $Q\cong \mathcal{Q}(\dis(Q),\dis(Q)_a,\widehat{L_a})$ and the fact that the core of the point stabilizer in a transitive permutation group is trivial. 
\end{proof}




In the following we denote by $\gamma_i(G)$ the $i$-th element of the lower central series of a group $G$.
\begin{Lem}\label{derived subgroup lemma}
Let $G$ be a finite group and $f\in Aut(G)$. If $G=[G,f]$ then $Fix(f)\leq \gamma_1(G)$.
\end{Lem}
\begin{proof}
Let $A=G/\gamma_1(G)$, the abelianization of $G$. Since $\gamma_1(G)$ is characteristic, the map $\partial f_A(a\gamma_1(G))=\partial f(a)\gamma_1(G)$ is well defined. Since $G=[G,f]$, $\partial f_A$ is surjective and then it is an automorphism of $A$. We have also that if $a\in Fix(f)$, then
\[\partial f_A(a\gamma_1(G))=\gamma_1(G).\]
Therefore $a\gamma_1(G)$ is in the kernel of $\partial f_A$ which is trivial and $a\in \gamma_1(G)$.
\end{proof}

\begin{The}[Derived subgroup Lemma]\label{stabilizer in derived}
Let $Q$ be a finite connected quandle. Then $\dis(Q)_a\leq \gamma_1(\dis(Q))$ for every $a\in Q$.
\end{The}
\begin{proof}
Let $a\in Q$, then $[\dis(Q),\widehat{\Lm{a}}]= \dis(Q)$ by Lemma \ref{commutator_with_trans_is_trans}. Now we apply Lemma \ref{derived subgroup lemma} to $G=\dis(Q)$, $f=\widehat{\Lm{a}}$, and conclude using the fact that $\dis(Q)_a\leq Fix(\widehat{\Lm{a}})$. 
\end{proof}
\begin{Cor}\label{on 2-step nilpotency}
Let $Q$ be a finite connected quandle. If $\dis(Q)$ is nilpotent of class $2$ then $Q$ is principal. 
\end{Cor}
\begin{proof}
Since $\dis(Q)$ is nilpotent of class $2$ we have that $\dis(Q)_{a}\leq \gamma_1(\dis(Q))\leq Z(\dis(Q))$. The stabilizers of a transitive action are conjugate, hence if the stabilizer of a point is a central subgroup, then it is trivial. Thus, $\dis(Q)_{a}=1$ and so $Q$ is principal (cfr. Lemma \ref{facts}(ii)).
%
\end{proof}
\begin{remark}
Let $Q$ be a connected quandle, $\dis(Q)=G$ and $\alpha=\mathcal{O}_{\gamma_2(G)}$. Then $\gamma_2(G)\leq \dis^\alpha$ and so $\dis(Q/\alpha)$ is nilpotent of class 2, thus $Q/\alpha$ is principal. We can therefore slightly improve the Derived subgroup Lemma. Indeed since $\alpha\leq \mathcal{O}_{\gamma_1(G)}$ then $\dis(Q)_a\leq \dis^\alpha\leq \gamma_1(G)$ for every $a\in Q$. In the \cite{RIG} library the quandles denoted by SmallQuandle(24,20) and SmallQuandle(24,21) provide examples of quandles in which $\dis^\alpha\lneqq \gamma_1(G)$.\end{remark}

The following application of Theorem \ref{stabilizer in derived} is the analogous of abelianization of groups for connected quandles. 
\begin{Pro}\label{ker of [1,1]}
Let $Q$ be a finite connected quandle and $G=\dis(Q)$. Then $\dis^{\gamma_1(Q)}=\gamma_1(G)$ and $Q/\gamma_1(Q)$ is affine over $G/\gamma_1(G)$.
\end{Pro}

\begin{proof}
Let $G=\dis(Q)$. From the definition of $\dis^{\gamma_1(Q)}$ given in \eqref{kernel}, it follows that the orbits of $\dis^{\gamma_1(Q)}$ and the orbits of $\gamma_1(G)$ acting on $Q$ coincide as by Proposition \ref{gamma_1} $\gamma_1(Q)=\mathcal{O}_{\gamma_1(G)}$. Let $a\in Q$, by virtue of Theorem \ref{stabilizer in derived}, $\dis(Q)_a\leq \gamma_1(G)\leq \dis^{\gamma_1(Q)}$ and 
\[|[a]|=\frac{|\gamma_1(G)|}{|\dis(Q)_a|}=\frac{|\dis^{\gamma_1(Q)}|}{|\dis(Q)_a|}.\]
Therefore $\gamma_1(G)=\dis^{\gamma_1(Q)}$ and so the factor $Q/\gamma_1(Q)$ is affine over $\dis(Q/\gamma_1(Q))\cong G/\gamma_1(G)$.
\end{proof}

Connectedness of a coset quandle over a nilpotent group $G$ can be established by looking at the induced structure of affine quandle over $G/\Phi(G)$ where $\Phi(G)$ is the Frattini subgroup of $G$. 
From now on we will use the following notation: if $G$ is a group, $f\in \aut{G}$ and $H\unlhd G$ such that $f(H)=H$,  $f_H$ will be the induced automorphism on $G/H$, i.e.:
\begin{equation}\label{induced_morphism}
f_H:G/H\longrightarrow G/H,\quad aH\mapsto f(a)H.
\end{equation}
\begin{Pro}\label{Q conn iff Q/[1,1,] conn} 
Let $G$ be a nilpotent group, let $Q=\Q(G,H,f)$ and $P=\aff(G/\Phi(G),f_{\Phi(G)})$. 
\begin{itemize}
\item[(i)] If $P$ is connected then $Q$ is connected and $\dis(Q)\cong G/Core_G(H)$.
\item[(ii)] If $P$ is connected and $G/\Phi(G)$ is cyclic then $Q$ is connected and affine over a cyclic group. 
\end{itemize}
\end{Pro} 
\begin{proof}
(i) The quandle $P$ is affine and connected and $\dis(P)$ is a regular subgroup of $G/\Phi(G)$. By Lemma \ref{commutator_with_trans_is_trans} $\dis(P)\cong [G/\Phi(G),f_{\Phi(G)}]$, therefore $[G/\Phi(G),f_{\Phi(G)}]=G/\Phi(G)$. The set $\setof{af(a)^{-1}\Phi(G)}{a\in G}$ generates $G/\Phi(G)$ and so $\setof{af(a)^{-1}}{a\in G}$ generates $G$. Therefore, $G=[G,f]$, and since the action of $[G,f]$ is the same of $\dis(Q)$,  $Q$ is connected and $\dis(Q)\cong G/Core_G(H)$.

(ii) By (i), $G=[G,f]$ and $G$ is cyclic, so $Q$ is connected and affine. 


\end{proof}

\begin{Cor}\label{nilpotent case}
Let $G$ be a finite $p$-group, $f\in \aut{G}$ and $H\leq Fix(f)$. If $f_{\Phi(G)}$ has no eigenvalues equal to $1$, then $Q=\Q(G,H,f)$ is connected.
\end{Cor}
\begin{proof}
Let $P=\aff(G/\Phi(G),f_{\Phi(G)})$. The group $G/\Phi(G)$ is an elementary abelian $p$-group, $Fix(f_{\Phi(G)})$ is the eigenspace relative to $1$. The affine quandle $P$ is connected if and only if it is latin, i.e. $Fix(f_{\Phi(G)})=1$ (see \ref{facts}(ii)). Thus, by Proposition \ref{Q conn iff Q/[1,1,] conn}(i), $Q$ is connected.
\end{proof} Corollary \ref{nilpotent case} can be extended to finite nilpotent groups by looking at each $p$-component in the factor with respect to the Frattini subgroup.

\section{Central congruences of faithful quandles}\label{central}

The main feature we are going to use to investigate nilpotent quandles, of which prime power size connected quandles are an instance (see Section \ref{p-racks}), is that central subgroups of transitive groups are semiregular. This fact becomes especially useful under the additional hypothesis that the quandle is faithful. 
\begin{Pro}\label{semiregul groups}
Let $Q$ be a finite faithful quandle, $a\in Q$ and let $N\in Norm(Q)$ be semiregular. Then:
\begin{itemize}
\item[(i)] $N=\dis_{\c{N}}=\{L_b L_a^{-1}, b \in [a]_{\c{N}}\} \cong \dis([a]_{\c{N}})$.

\item[(ii)] $\c{N}=\mathcal{O}_N$ and $[a]_{\c{N}}$ is a connected subquandle.

\item[(iii)]  
$\dis(Q)_{[a]_{\c{N}}} \cong \dis_{\c{N}}\rtimes \dis(Q)_ a$ and $\dis^{\c{N}}\cong \dis_{\c{N}}\rtimes \dis^{\c{N}}_a$.
\end{itemize}
\end{Pro}
\begin{proof}
Let $[a]$ be the block of $a$ in $\c{N}$.

(i) The mapping:
\begin{equation}
\varphi :N\longrightarrow N, \quad n\mapsto n L_a n^{-1} L_a^{-1}= L_{n(a)} L_{a}^{-1},
\end{equation}
is bijective. Indeed, since $Q$ is faithful $\varphi(n)=\varphi(m)$ if and only if $n(a)=m(a)$, and since $N$ is semiregular then $n=m$. Therefore $N=\setof{L_{n(a)} L_a^{-1}}{n\in N} \leq \dis_{\c{N}}  $ for every $a\in Q$. Since it is always the case that 
 $ \dis_{\c{N}} \leq N$, we can conclude that $N=\dis_{\c{N}} \cong \dis([a])$, where the last isomorphism follows from $N$ being semiregular.

(ii) By 
(i) if $b\, \c{N}\, a$ (i.e. $L_b L_a^{-1}\in N$) then $b=n(a)$ for some $n\in N$. Thus the orbit of $a$ under $N$ coincides with $[a]$. Therefore the group $N \cong \dis([a])$ is transitive on $[a]$, and the block of $a$ is a connected subquandle.

(iii) By 
(ii) $N=\dis_{\c{N}}$ is regular on $[a]$, i.e. $|[a]|=|\dis_{\c {N}}|$,  and since $\dis_{\c{N}}$ is included in $\BlocS{Q}{a}$ we have that $\BlocS{Q}{a}$ is transitive on $[a]$.  The index of $\dis(Q)_a$ in $\BlocS{Q}{a}$ is then $|[a]|$, $\dis_{\c{N}}$ is normal and $\dis_{\c{N}}\bigcap \dis(Q)_a=1$. We can conclude that $\dis_{\c{N}}$ and $\dis(Q)_a$ generate the whole block stabilizer which splits as a semidirect product of the two subgroups. The same argument shows that a similar decomposition holds for $\dis^{\c{N}}$.
\end{proof}

\begin{Cor}\label{strucure_of_K_N}
Let $Q$ be a finite faithful connected quandle, $a\in Q$ and $\alpha$ be a central congruence. Then: 
\begin{itemize}
\item[(i)] $\alpha=\c{\dis_\alpha}$ and $[a]_{\alpha}$ is an affine latin quandle.
\item[(ii)] $\BlocS{Q}{a}\cong \dis_\alpha\times \dis(Q)_a$ and $\dis^\alpha \cong \dis_\alpha\times \dis^\alpha_a$.
\end{itemize}
\end{Cor}

\begin{proof}
Let $[a]$ be the block of $a$ in $\alpha$. The subgroup $\dis_\alpha$ is central (Theorem \ref{central cong}) and therefore semiregular. According to Proposition \ref{semiregul groups}(ii), $\dis_\alpha$ is regular on each block of $\c{\dis_\alpha}$. Therefore, $\c{\dis_\alpha}=\mathcal{O}_{\dis_\alpha}\leq \alpha$, hence $\alpha$ coincides with $\c{\dis_\alpha}$ and so we we can apply Proposition \ref{semiregul groups} to $\alpha$.

(i) The block $[a]$ is an affine latin quandle since it is connected and $\dis_\alpha \cong \dis([a])$ is abelian by Proposition \ref{semiregul groups} (ii).

(ii) The congruence $\alpha$ is central and then by Theorem \ref{central cong} $\dis(Q)_a=\dis(Q)_b$ whenever $a\,\alpha\, b$. Using this, if $h\in \BlocS{Q}{a}$, then $h\dis(Q)_a h^{-1}=\dis(Q)_{h(a)}=\dis(Q)_a$ and so $\dis(Q)_a$ is normal in $\BlocS{Q}{a}$. We can now apply Proposition \ref{semiregul groups}(iii) and conclude that the product is direct. The same argument holds for $\dis^\alpha$.
%
\end{proof}

\begin{Cor}\label{embedding of ker}
Let $Q$ be a finite faithful connected quandle and $\alpha$ be a central congruence. Then $\dis^\alpha$ embeds into $\dis([a]_\alpha)^{n}$, where $n$ is the number of generators of $Q/\alpha$. In particular, $\dis^\alpha$ is abelian and $|\dis^\alpha|$ divides $|[a]_\alpha|^n$.
\end{Cor}
\begin{proof}
The mapping
\begin{equation}\label{map_1}
\dis^\alpha\longrightarrow \prod_{[a]\in Q/\alpha} \aut{[a]}, \quad h\mapsto \setof{h|_{[a]}}{[a]\in Q/\alpha}
\end{equation}
is a group homomorphism, and according to Proposition \ref{strucure_of_K_N}(ii) $\dis^\alpha\cong \dis_\alpha\times \dis^\alpha_{a}$. Therefore $\dis^\alpha|_{[a]}=\dis_\alpha|_{[a]}\cong  \dis([a])$, so the image of \eqref{map_1} is contained in $\prod_{[a]\in Q/\alpha} \dis([a])$. If $Q/\alpha$ is generated by $[a_1],\ldots,[a_n]$, then $Q$ is generated by $\bigcup_{i=1}^n [a_i]$. So if $h|_{[a]_i}=1$ for every $i=1,\ldots,n$, then $h=1$. Therefore the mapping in \eqref{map_1} is injective. The blocks of $\alpha$ are affine and connected, so $|[a]|=|\dis([a])|$, therefore $|\dis^\alpha|$ divides $|[a]|^n$.
\end{proof}
Corollary \ref{embedding of ker} does not hold for abelian congruences, e.g. for SmallQuandle(12,10) in the \cite{RIG} library.

In \cite[Proposition 7.8]{CP} we extended a result on representation of Mal'cev algebras in terms of their factor with respect to the center. We say that $Q$ is a {\it central extension} of $Q/\alpha$ if there exists an abelian group $A$, $\psi\in \aut{A}$ and $\theta:Q/\alpha\times Q/\alpha\longrightarrow A$ a mapping satisfying some additional conditions, such that $Q\cong (Q/\alpha\times A,\ast)$, where 
\begin{equation}\label{central ext}
(a,s)\ast (b,t)=(a\ast b, (1-\psi)(s)+\psi(t)+\theta_{a,b}),
\end{equation}
for every $a,b\in Q/\alpha$ and every $s,t\in A$.\\
This construction applies to finite faithful connected quandles. 

\begin{Lem}\label{faithful implies latin}
Let $Q$ be a finite faithful connected quandle and $\alpha$ be a central congruence. Then $Q$ is a central extension of $Q/\alpha$. If $Q/\alpha$ is latin, then $Q$ is latin.
\end{Lem}
\begin{proof}
By Lemma \ref{strucure_of_K_N}, $[a]_\alpha$ is a latin quandle and $\dis_\alpha$ is regular on each block of $\alpha$. Since $Q$ is connected, $Q/\alpha$ is connected. So we can apply Proposition 7.8 of \cite{CP} to deduce that $Q$ is a central extension of $Q/\alpha$ as in \eqref{central ext}. Since $[a]_{\alpha}$ is a latin quandle, $1-\psi\in \aut{A}$. Hence if $Q/\alpha$ is latin, then 
$$R_{(b,t)}(a,s)=(R_b(a),(1-\psi)(s)+\psi(t)+\theta_{a,b}),$$
is bijective for every $(b,t)\in Q/\alpha\times A$, and then $Q$ is latin.
\end{proof}

\begin{Pro}\label{trans of 2step nil}
Let $Q$ be a finite faithful connected quandle of nilpotency length $2$. Then $Q$ is latin and $\dis^{\zeta_Q}\cong Z(\dis(Q))\times Fix(\widehat{L}_a)$ for every $a\in Q$.
 \end{Pro}
 \begin{proof} The center of $Q$ is $\c {Z(\dis(Q))}$ by \eqref{zeta_Q_Con_zeta_Q}. The factor $Q/\zeta_Q$ is finite, abelian and connected, hence  $\dis(Q/\zeta_Q)$ is abelian (Theorem \ref{central cong}). In particular $Q/\zeta_Q$ is affine and latin \cite[Theorem 7.3]{hsv}. According to Proposition \ref{faithful implies latin}, $Q$ is latin. 
 The stabilizer of the block $[a]$ coincides with the kernel $\dis^{\zeta_Q}$, indeed
 \begin{equation}\label{block_stab=kernel}
\BlocS{Q}{a}=\pi^{-1}_{\zeta_Q}(\dis(Q/\zeta_Q)_{[a]})=\pi^{-1}_{\zeta_Q}(1)=\dis^{\zeta_Q}.
\end{equation}
By Corollary \ref{semiregul groups} then $\dis_{\zeta_Q}=Z(\dis(Q)$ and now we can conclude by Corollary \ref{strucure_of_K_N}(ii) that $\dis^{\zeta_Q}\cong Z(\dis(Q))\times \dis(Q)_a=Z(\dis(Q))\times Fix(\widehat{L}_a)$. 
 \end{proof}
\section{Quandles of prime power order}\label{p-racks}
\subsection*{The \disp\text{} group and its center}
We now turn our attention to connected quandles of order a power of a prime and we call them {\it p-quandles}. They have a minimal coset representation as $\Q(\dis(Q),\dis(Q)_a,\widehat{\Lm{a}})$ (see \eqref{min rep}). In this case $\dis(Q)$ is a $p$-group according to \cite[Corollary 5.2]{GiuThe} and $Q$ is nilpotent \cite[Corollary 6.6]{CP}. 

%
We will exploit the fact that every p-group has a non-trivial center. The center, $Z(\dis(Q))$, is a non trivial semiregular subgroup which is normal in $\lmlt{Q}$, hence its order must divide $|Q|$ and hence $| Z(\dis(Q)) |=p^i$ with $1 \leq i \leq n$. Our aim is to find some bound on the order of $\dis(Q)$ by discussing the possible orders of $Z(\dis(Q))$. In turn any bound found can be translated in properties of the quandle via its minimal representation.\\ 
\noindent Let us start by seeing what happens when $| Z(\dis(Q)) | = | Q |$.
\begin{Pro}
\label{order_of_Z=Q_iff_Q_affine} Let $Q$ a finite connected quandle. Then $Q$ is affine if and only if $| \Z | = | Q |$.
\end{Pro}
\begin{proof}
If $Q$ is affine, then $\dis(Q)=Z(\dis(Q))$ and it is regular on $Q$, i.e. $|Q|=|Z(\dis(Q)|$. On the other hand, if $|Z(\dis(Q))|=|Q|$, then $Z(\dis(Q))$ is regular on $Q$ and then it provides an affine representation for $Q$.
%
%
%
\end{proof}
It is a well known fact that if a group is cyclic over its center then it is abelian. Using the following Proposition we get a similar result for faithful connected nilpotent quandles.
\begin{Pro}\label{generators_of_trans} 
Let $Q$ be a faithful connected nilpotent quandle. If $ \zeta_1,\ldots,\zeta_n$ generate $Z(\dis(Q))$ and $\psi_1 \dis^{\zeta_Q},\ldots ,\psi_m\dis^{\zeta_Q} $ generate $\dis(Q/\zeta_Q)$ then $\zeta_1,\ldots,\zeta_n$, $\psi_1,$ $\ldots,\psi_m$ generate $\dis(Q)$.
\end{Pro}

\begin{proof}
By Corollary \ref{strucure_of_K_N}, if $\xi_{1}, \dots ,\xi_{t}$ generate $\dis(Q)_a$ then $\psi_{1}, \dots,\psi_{m}$,  $\zeta_{1}, \dots, \zeta_{n},\xi_{1}$,  $\dots ,\xi_{t}$ generate $\dis(Q)$. Now $\dis(Q)$ is a nilpotent group and by Corollary \ref{stabilizer in derived}  $\dis(Q)_a\leq \gamma_1(\dis(Q)) \leq \frattini (\dis(Q))$. So, $\dis(Q)$ is generated by $\zeta_1,\ldots,\zeta_n,\psi_1,\ldots,\psi_m$.
\end{proof}

\begin{Cor}
\label{quotient_cannot_be_cyclic} Let $Q$ be a finite faithful connected nilpotent quandle. If $\dis(Q/\zeta_Q)$  is cyclic then $Q$ is affine.
\end{Cor}
%

%
%
%
\subsection*{Classification of connected Quandles of Order p and p$^2$}
%
 Simple quandles are classified by Andruskiewitsch and Gra\~{n}a in \cite{AG2003}, in particular in the prime power size case. The classification of quandles of size $p$ and $p^2$ is already known \cite{EGS,Grana_p2}. Let us give here a self contained proof of the same results, using a different approach, based on the analysis of the subgroups of $\dis(Q)$ normal in $\lmlt{Q}$ and the conditions on the center of $\dis(Q)$.
%
%
%
\begin{The}[{\cite[Theorem 3.9]{AG2003}}]
\label{classification_simple_quandles}Every simple quandle $Q$ of order $p^{n}$, $p$ a prime, is isomorphic to an affine latin quandle $\aff ( \mathbb{Z}_{p}^{n}, f )$ where $f$ acts irreducibly.
\end{The}
\begin{proof}
If $Q$ is simple then $\dis(Q)$ is a minimal normal subgroup of $\lmlt{Q}$, hence  $\dis(Q)$ being a $p$-group, it must be elementary abelian. Since it is abelian and transitive it is regular and has order $p^{n}$ and then has rank $n$. Subgroups of $\dis(Q)$ invariant under $\widehat{\Lm{a}}$ provide congruences of $Q$. Therefore, $Q \cong \aff(  \mathbb{Z}_{p}^{n},f)$ and $f$ acts irreducibly. 
\end{proof}
The classification of simple quandles of order a power of a prime yields immediately a characterisation of connected quandles of prime order.
\begin{The}[{\cite[Lemma 3]{EGS}}]
\label{quandles_of_prime_order} Connected quandle of prime order are affine.
\end{The} 
\begin{proof}
The quandle $Q$ is simple of size $p$ since the size of any factor of $Q$ divides $p$. By Theorem \ref{classification_simple_quandles} it is affine. \end{proof}

If $Q$ is a connected quandle, then $Z(\dis(Q))$ is a semiregular group and then $|Z(\dis(Q))|\leq |Q|$. For quandles of prime power size we have the following.
\begin{Pro}
\label{order_of_Z_cannot_be_p^n-1}Let $Q$ be a connected quandle of order $p^n$. Then the order of $\Z$ is not $p^{n-1}$.
\end{Pro}
\begin{proof}
The size of factors of a connected quandle divides the size of the quandle. The size of the orbits of $Z(\dis(Q))$ equals the size of $Z(\dis(Q))$ and they are contained in the blocks of $\zeta_Q$. Suppose that $Z(\dis(Q))$ has order $p^{n-1}$, so $\dis(Q)$ is not abelian by Proposition \ref{order_of_Z=Q_iff_Q_affine} and then $\zeta_Q\neq 1_Q$. Therefore $|Z(\dis(Q))|=p^{n-1}\leq|[a]_{\zeta_Q}|<p^n$. So $[a]_{\zeta_Q}$ coincides with the orbit of $a$ and $Q/\zeta_Q$ has size $p$ and so it has cyclic \disp\, group \ref{quandles_of_prime_order}. Moreover $Z(\dis(Q))\leq \dis^{\zeta_Q}$ and $Z(\dis(Q))\bigcap \dis^{\zeta_Q}_a=1$, then by cardinality $\dis^{\zeta_Q}= Z(\dis(Q))\times \dis^{\zeta_Q}_a$. By Corollary \ref{stabilizer in derived}, $\dis^{\zeta_Q}_a\leq \dis(Q)_a\leq \gamma_1(\dis(Q))\leq \Phi(\dis(Q))$ because $\dis(Q)$ is a $p$-group, hence $\dis(Q)$ is cyclic over its center and then abelian, contradiction.
\end{proof}

Using Proposition \ref{order_of_Z_cannot_be_p^n-1} we can obtain the classification of quandles of prime squared order and show that, similarly to what happens for groups, they are abelian and then affine.
\begin{The}[{\cite[Proposition 3.10]{Grana_p2}}]
\label{quandles_of_prime_squared_order}Connected quandles of prime squared order are affine.
\end{The}
\begin{proof}
Let $Q$ be a connected quandle of size $p^2$. According to Proposition \ref{order_of_Z_cannot_be_p^n-1}, $|Z(\dis(Q))|=p^2$, since it is not trivial and therefore $Q$ is affine.%
\end{proof}

\section{Connected Quandles of Prime Cube Order}\label{cube order}
We now turn to the classification of connected quandles of prime cube order.  By Proposition \ref{order_of_Z_cannot_be_p^n-1}, either $|\Z | = p^3$ or $|\Z|=p$. The first case is clear by Proposition \ref{order_of_Z=Q_iff_Q_affine}, and a classification can be found in \cite{Hou}. In this section we will discuss the second case and we can assume that such quandles are not simple by Theorem \ref{classification_simple_quandles}.  We are going to use that the non-simple ones have a central congruence with an affine factor of size $p^2$.
\begin{Pro}\label{2nilp}
Let $Q$ be a connected quandle of size $p^3$ with $|Z(\dis(Q))|=p$. Then $Q$ is nilpotent of length $2$ and $\gamma_1(Q)=\mathcal{O}_{Z(\dis(Q))}$ is the smallest congruence of $Q$. If $Q$ is faithful, then $\gamma_1(Q)=\zeta_Q$.
\end{Pro}
\begin{proof}
Let $Q$ be a connected quandle of size $p^3$. The blocks of $\beta= \mathcal{O}_{Z(\dis(Q))}$ have size $p$ and therefore $\beta$ is a minimal congruence. On the other hand, every factor of $Q$ has size at most $p^2$ for every $\alpha \in Con(Q)$ and therefore it is abelian. Hence $\gamma_1(Q)$ is the smallest congruence of $Q$ and so $\beta=\gamma_1(Q)\leq \zeta_Q$, i.e. $Q$ is nilpotent of length $2$. If $Q$ is faithful $\zeta_Q=\c{Z(\dis(Q)}$ (see \ref{zeta_Q_Con_zeta_Q}). So we can apply Corollary \ref{strucure_of_K_N} and the blocks of $\zeta_Q$ are connected, so $Z(\dis(Q))=\dis_{\zeta_Q}$ is transitive over each block, i.e. $\zeta_Q=\mathcal{O}_{Z(\dis(Q))}$.
\end{proof}

A immediate consequence of Proposition \ref{2nilp} and Proposition \ref{faithful implies latin} is the following: 
\begin{Cor}
Let $Q$ be a connected quandle of size $p^3$. If $Q$ is faithful then $Q$ is latin.\qed
\end{Cor}
In the light of the last Corollary the two cases to be discussed are non-faithful and latin.

\subsection*{Non-faithful case}
A characterization of non-faithful connected quandles of size $p^3$ is given at the end of this section in Theorem \ref{non-faithful p to the cube}. First we show that such quandles are principal over the group $\mathbb{Z}_p^2\rtimes \mathbb{Z}_p$.
\begin{Pro}\label{cov of principal}
Let $Q$ be a connected quandle and let $Q/\lambda_Q$ be a principal quandle. Then $Q$ is a principal quandle.
\end{Pro}
\begin{proof}
If $Q/\lambda_Q$ is principal then $\dis^{\lambda_Q}=\BlocS{Q}{a}$. By virtue of Proposition 2.49 of \cite{Eisermann}, $\dis^{\lambda_Q}$ is central. Since $Q$ is connected, then $\dis^{\lambda_Q}$ is transitive on the blocks of $\lambda_Q$ (Proposition \ref{connected ext}) and it contains $\dis(Q)_a$ according to \eqref{stabiliser in block stabiliser}. We can infer that $\dis^{\lambda_Q}$ is abelian and transitive, therefore regular, i.e. $ \dis(Q)_a=1$ and $Q$ is principal.
\end{proof}
Recall that there exist two isomorphism classes of non abelian groups of size $p^3$, namely $\mathbb{Z}_{p^2}\rtimes \mathbb{Z}_p$ and $\mathbb{Z}_{p}^2\rtimes \mathbb{Z}_p$. 
\begin{Lem}\label{on principal case}
Let $Q$ be a connected principal quandle of size $p^3$ with $|Z(\dis(Q))|=p$. Then $\dis(Q)\cong \mathbb{Z}_p^2\rtimes \mathbb{Z}_p$.
\end{Lem}
\begin{proof}
Assume that $\dis(Q)$ is isomorphic to $G=\mathbb{Z}_{p^2}\rtimes \mathbb{Z}_p$ and $f\in \aut{G}$. According to \cite[Corollary 1]{winter1972} the induced automorphism on the factor $G/Z(G)$ has a fixed point. So $R=\aff(G/Z(G),f_{Z(G)})$ is a non connected factor of $Q$ according to \ref{facts}(ii) $R$ is not connected and then $Q$ is not connected. 
%
\end{proof}

\begin{The}[Charaterization of non-faithful connected $p^3$-quandles\label{non-faithful p to the cube}]
Let $Q$ be a quandle of size $p^3$. The following are equivalent:
\begin{itemize}
\item[(i)] $Q$ is connected and non-faithful with $|Z(\dis(Q))|=p$. 
\item[(ii)] $ Q\cong \Q(\mathbb{Z}_p^2\rtimes \mathbb{Z}_p,f)$, $Z(\dis(Q))=Fix(f)$ and $Fix(f_{\gamma_1(G)})=1$ .
\end{itemize}
\end{The}
\begin{proof}
Note that for a principal quandle $\Q(G,f)$ the blocks of $\lambda_Q$ correspond to cosets with respect to $ Fix(f)$ and the orbits with respect to the action of a group $N\leq G$ are the cosets with respect to $N$.

(i) $\Rightarrow$ (ii) The factor $Q/\lambda_Q$ is affine, thus by Proposition \ref{cov of principal} $Q$ is principal and by Lemma \ref{on principal case} $Q\cong \Q( \mathbb{Z}_p^2\rtimes \mathbb{Z}_p,f)$. Every factor of $Q$ is connected and affine, then it is faithful and so $\lambda_Q$ is a the smallest congruence of $Q$ (indeed if $L_a=L_b$ then $L_{[a]_\alpha}=L_{[b]_\alpha}$ for every congruence $\alpha$ and therefore $[a]_\alpha=[b]_\alpha$). By Proposition \ref{2nilp}, then $\lambda_Q=\mathcal{O}_{Z(\dis(Q))}$ . Hence $Fix(f)=Z(\dis(Q))$. The factor $Q/\gamma_1(Q)=\aff(G/\gamma_1(G),f_{\gamma_1(G)})$ is connected and then latin, so $Fix(f_{\gamma_1(G)})=1$ (\ref{facts}(ii)).

(ii) $\Rightarrow$ (i) For the group $G=\mathbb{Z}_p^2\rtimes \mathbb{Z}_p$ we have that $\gamma_1(G)=\Phi(G)$ and its affine factor $\aff(G/\Phi(G),f_{\Phi(G)})$ is connected since $Fix(f_{\Phi(G)})=1$. By virtue of Corollary \ref{nilpotent case}, $Q$ is connected and principal. Since $Fix(f)\neq 1$ then $Q$ is not faithful.
\end{proof}


\subsection*{Latin Case}
%
%
In Proposition \ref{Lem nec cond} we will find some necessary and sufficient conditions to be satisfied by pairs $(G,f)$ where $G$ is a group and $f\in \aut{G}$ for which $Q=\Q(G,Fix(f),f)$ is a latin quandle of size $p^3$ with $\dis(Q)\cong G$ and $|Z(\dis(Q))|=p$. Recall that for faithful quandles $\dis(Q)_a=Fix(\widehat{\Lm{a}})$. To this aim we will make use of Proposition \ref{sims}:
\begin{Pro}[{\cite[Proposition 9.2.5]{Sims}}]
\label{sims}Let $G$ be a group and let $X\subseteq G$ such that $G/\gamma_1(G)$ is generated by the set $\setof{ x\gamma_1(G)}{x\in X}$. Then $\gamma_1(G)/\gamma_2(G)$ is generated by the set of their commutators, i.e. $\setof{[x,y]\gamma_2(G)}{x,y\in X}$. \qed 
\end{Pro}

\begin{Pro}\label{Lem nec cond}
Let $Q$ be a latin quandle of order $p^3$, $G=\dis(Q)$, and $|Z(G)|=p$.  Then one of the following holds:
\begin{itemize}
\item[(i)] $|G|=p^3$, $Z(G)=\gamma_1(G)\cong\mathbb{Z}_p$ and $Fix(\widehat{L_a})=1$.
\item[(ii)]  $|G|= p^4$, $\gamma_1(G)\cong Z(G)\times Fix(\widehat{\Lm{a}})\cong \mathbb{Z}_p^2$ and in particular $G$ is  nilpotent of length $3$.
\end{itemize}
\end{Pro}

\begin{proof}
By Proposition \ref{2nilp} $\zeta_Q=\gamma_1(Q)=\c{Z(\dis(Q))}$ (since $Q$ is faithful, see \eqref{zeta_Q_Con_zeta_Q}), by Proposition \ref{strucure_of_K_N} we have that $\dis_{\zeta_Q}=Z(\dis(Q))$ and it has size $p$. Using Propositions \ref{ker of [1,1]} and \ref{trans of 2step nil}  $$\gamma_1(G)\cong Z(G)\times Fix(\widehat{\Lm{a}})\cong \mathbb{Z}_p^k,$$ for some $k$, where the last isomorphism follows from Corollary \ref{embedding of ker}.\\ Moreover $\dis(Q/\zeta_Q)\cong G/\gamma_1(G)$ can not be cyclic, otherwise $G/\Phi(G)$ would be cyclic, since $\gamma_1(G)\leq \Phi(G)$, and by virtue of Proposition \ref{Q conn iff Q/[1,1,] conn}(ii) $\dis(Q)$ would be cyclic as well. Therefore $\dis(Q/\zeta_Q)\cong \mathbb{Z}_p\times \mathbb{Z}_p$ and $\gamma_2(G)=[\gamma_1(G),G]\leq \dis_{\zeta_Q}=Z(G)$ ($[\dis^\alpha,\dis(Q)]\leq \dis_\alpha(Q)$ for every $\alpha\in Con(Q)$, see \cite[Proposition 3.3 (1)]{CP}). So either $\gamma_2(G)=Z(G)$ or $\gamma_2(G)=1$. This dichotomy delivers us the two cases we are aiming at:
\begin{itemize}
\item[(i)] if $\gamma_2(G)=1$, then according to Corollary \ref{on 2-step nilpotency} $Q$ is principal and so $Fix(\widehat{\Lm{a}})=1$ since $Q$ is latin (\ref{facts}(ii)).
\item[(ii)] If $\gamma_2(G)=Z(\dis(Q))$, then $\gamma_1(Q)/\gamma_2(Q)$ is cyclic by Proposition \ref{sims} and therefore $|\dis(Q)|=|\dis(Q/\zeta_Q)||Z(G)||\gamma_1(G)/\gamma_2(G)|=p^4$ and $G$ is nilpotent of length $3$.
\end{itemize}
\end{proof}
Using Proposition \ref{on principal case} we can provide a characterization of the \disp\, group in case (i) of Proposition \ref{Lem nec cond}. 
\begin{The}[Charaterization of faithful connected principal $p^3$-quandles]
\label{principal latin}Let $Q$ be a quandle of size $p^3$. The following are equivalent:
\begin{itemize}
\item[(i)] $Q$ is a principal latin quandle with $|Z(\dis(Q))|=p$. 
\item[(ii)] $Q\cong \Q(\mathbb{Z}_p^2\rtimes \mathbb{Z}_p,f)$ with $Fix(f)=1$.
\end{itemize} \qed
\end{The}

Let us consider case (ii) of Proposition \ref{Lem nec cond}. The group $G=\dis(Q)$ and the automorphism $f=\widehat{L_a}$ need to satisfy some necessary conditions, namely:
\begin{align}
& \gamma_2(G)= Z(G),\qquad\qquad \gamma_1(G)\cong \mathbb{Z}_p\times \mathbb{Z}_p \tag{G},\label{condition_on_G}\\
& \gamma_1(G)\cong Fix(f)\times Z(G) \tag{A}\label{condition_on_aut}
\end{align}

We are going to show that conditions \eqref{condition_on_G} and \eqref{condition_on_aut} are also sufficient for the pair $(G,f)$ to provide a latin quandles of size $p^3$ as in case (ii) of Proposition \ref{Lem nec cond}. 
The classification of groups of order $p^4$ is known and by a direct inspection of their presentations provided in \cite{tedesco} it is straightforward to verify the following:
\begin{Pro}
\label{admissible groups}
Let $G$ be a group of size $p^4$. Then $G$ satifies \eqref{condition_on_G} if and only if $G$ is one of the following groups:
\begin{align*}
G_7&=\langle g_1, g_2, g_3, g_4, \, | \, [g_2,g_1]=g_3, \, [g_3,g_1]=g_4\rangle,\\
G_8&=\langle g_1, g_2, g_3, g_4, \, | \, [g_2,g_1]=g_3, \, [g_3,g_1]=g_1^p=g_4\rangle.\\
G_9 &=\langle g_1, g_2, g_3, g_4, \, | \, [g_2,g_1]=g_3, \, [g_3,g_1]=g_2^p=g_4\rangle,\\
G_{10} &=\langle g_1, g_2, g_3, g_4, \, | \, [g_2,g_1]=g_3, \, [g_3,g_1]=g_2^p=g_4^w\rangle,
\end{align*}
where $w$ is any primitive root of unity. \qed
\end{Pro}
%
 The details of the description of the automorphisms of the groups $G_i$, $7\leq i\leq 10$ we are using in the next of the proof can be found in \cite[Section 5.4-5.8]{tedesco}).
\begin{The}[Charaterization of faithful connected non-principal $p^3$-quandles]
\label{P_alla4}Let $G_i$ be a group with $i=7\ldots 10$ and $p>3$. The following are equivalent:
\begin{itemize}
\item[(i)] $Q=\Q(G_i,Fix(f),f)$ is a latin quandle of size $p^3$ with $\dis(Q)\cong G_i$;
\item[(ii)] $(G_i,f)$ satisfy condition \eqref{condition_on_aut}.
\end{itemize}\end{The}

\begin{proof}
The implication (i) $\Rightarrow$ (ii) is Proposition \ref{Lem nec cond}. In order to prove (ii) $\Rightarrow$ (i) we will prove that $Q$ is a non-principal connected quandle: indeed, if $Q$ is non-principal, then $Q$ is faithful (\ref{non-faithful p to the cube}) and so also latin (\ref{faithful implies latin}).
By virtue of Corollary \ref{nilpotent case} in order to show that $Q$ is connected it is enough to show that $f_{\Phi(G)}$ has no eigenvalue equal to $1$. Note that $\gamma_1(G_i)=\Phi(G_i)$ for every $i=7,\ldots 10$.

Case $i=7$: the subgroup $\Phi(G_7)$ is elementary abelian, so the restriction of $f$ to $\Phi(G_7)$, is represented by a matrix with respect to a basis $a_3,a_4$ with $a_4\in Z(G_7)$. So we have 
\begin{displaymath}
f|_{\Phi(G_7)}=
\begin{bmatrix}
u_1v_2 & 0 \\  
  u_1 v_3+\frac{u_1-1}{2}u_1 v_2 & u_1^2 v_2 \\
  \end{bmatrix},\quad  f_{\Phi(G_7)}=\begin{bmatrix}
u_1& 0\\
u_2 & v_2
\end{bmatrix}.
  \end{displaymath}
The eigenvalues of matrix $f|_{\Phi(G)}$ are $\lambda_1= u_1v_2$ and $\lambda_2=u_1^2 v_2$ and $Z(G_7)$ is the eigenspace relative to $\lambda_2$. If the pair $(G_7,f)$ satisfies the contidion \eqref{condition_on_aut} then one of the eigenvalues of $M$ is 1 and $\gamma_1(G_7)=\Phi(G_7)$ is the product of the center and of the eigenspace relative to $1$. Then $\lambda_1=1$, i.e. $u_1 v_2=1$ and $\lambda_2= u_1^2 v_2=u_1\neq 1$ and then $v_2\neq 1$. The eigenvalues of $f_{\Phi(G)}$ are $u_1$ and $v_2$, both different from $1$.

A similar argument will do in the other cases for which we just list the relevant matrices. 

Case $i=8$: 
\begin{displaymath}
f|_{\Phi(G_8)}
=\begin{bmatrix}
1 &  0\\
u_1 u_3+\frac{u_1-1}{2} & u_1
\end{bmatrix},\quad f_{\Phi(G_8)}=\begin{bmatrix}
u_1 &  0\\
u_2 & u_1^{-1}
\end{bmatrix}.
\end{displaymath}

Case $i=9,10$: these two cases can be treated together, since the description of automorphisms of the two groups in terms of their actions on the relative generators is the very same. 
\begin{displaymath}
f|_{\Phi(G_i)}
=\begin{bmatrix}
\pm v_2 &  0\\
\pm v_3 & v_2
\end{bmatrix},\quad f_{\Phi(G_i)}=\begin{bmatrix}
\pm 1 &  0\\
0 & v_2
\end{bmatrix}.
\end{displaymath}
In all the cases the center of $G_i$ is the unique normal subgroup of $G_i$ contained in $\gamma_1(G_i)$. Therefore $Fix(f)$ is not a normal subgroup, so $Core_{G_i}(Fix(f))=1$ and $\dis(Q)\cong G_i$  (Proposition \ref{Q conn iff Q/[1,1,] conn}).
\end{proof}

\subsection*{Isomorphism classes}
In order to compute the isomorphism classes of $p^3$-quandles we will be using a particular instance of the following isomorphism theorem for connected quandles.
\begin{The}[Isomorphism Theorem for connected coset quandles]\label{iso theorem}
Let $G$ be a group, $f_1, f_2\in \aut{G}$ and $Q_i=\Q(G,H_i,f_i)$ and $\dis(Q_i)= G$. Then $Q_1\cong Q_2$ if and only if there exists $h\in \aut{G}$ such that $f_2=h^{-1} f_1 h$ and $h(H_2)=H_1$.
\end{The}

\begin{proof}
If such $h\in \aut{G}$ exists, then $\widehat{h}(aH_1)=h(a)H_2$ is a quandle isomorphism between $Q_1$ and $Q_2$. On the other hand, if $\phi: Q_1\longrightarrow Q_2$ is a quandle isomorphism, the mapping defined as
\begin{displaymath}
\pi_\phi:\dis(Q_1)\longrightarrow \dis(Q_2),\quad L_a L_b\inv \mapsto L_{\phi (a)} L_{\phi(b)}^{-1},
\end{displaymath}
for every $a,b\in Q_1$, is a well-defined group isomorphism with the desired properties.
\end{proof}

\begin{Cor}\label{corollary iso theo}
Let $G$ be a group, $f_1, f_2\in \aut{G}$ and $Q_i=\Q(G,H_i,f_i)$ and $\dis(Q_i)= G$. If $Q_1$ and $Q_2$ are both faithful or principal then $Q_1\cong Q_2$ if and only if there exists $h\in \aut{G}$ such that $f_2=h^{-1} f_1 h$.\qed
\end{Cor}

\begin{proof}
	If $Q_1$ and $Q_2$ are faithful (resp. principal) then $H_i=Fix(f_i)$ for $i=1,2$ (resp. $H_1=H_2=1$). So if $f_1$ and $f_2$ are conjugate the second condition of Theorem \ref{iso theorem} is satisfied.
\end{proof}

According to Corollary \ref{corollary iso theo}, isomorphism classes of connected quandles of size $p^3$ are parametrized by pairs $(G,[f]_\sim)$ where $[f]_\sim$ denotes the conjugacy class of $f$ in $\aut{G}$. Theorem \ref{principal latin} and Theorem \ref{P_alla4} give necessary and sufficient conditions on the pairs $(G,f)$ in order to provide a quandle with the desired properties. We are now able to give an explicit description of isomorphism classes of connected quandles of size $p^3$ with $p>3$ and we refer to Section \ref{appendix} for the explicit computations of conjugacy classes.

The automorphisms of the group $\mathbb{Z}_p^2\rtimes \mathbb{Z}_p=\langle g_1,g_2\rangle$ are described by their action on the two generators. Table \ref{Tab2} collects isomorphism classes of (non affine) principal connected quandles.
\begin{table}[!htb]
\caption{Principal connected quandle of size $p^3$.}
\label{Tab2}
\begin{tabular}{|l|l|l|c|c|}
\hline
  & Parameters & Latin & \#  \\
 \hline
 $g_1\mapsto g_1^\lambda$, \quad  $g_2\mapsto g_2^\mu$ & $\lambda \leq \mu \in \{2,\ldots, p-1\}$, $\lambda \mu \neq 1$& yes & $\frac{(p-1)(p-3)}{2}$\\
 \hline
 $g_1\mapsto g_1^\lambda$, \quad  $g_2\mapsto g_1 g_2^\lambda$ & $\lambda \in \{2,\ldots, p-2\}$ & yes & $p-3$\\
 \hline
  $g_1\mapsto g_2$, \quad  $g_2\mapsto  g_1^{-a} g_2^{-b}$ & $x^2+bx+a$ is an irreducible polynomial and  $a\neq 1$ & yes & $\frac{(p-1)^2}{2}$\\
 \hline
 $g_1\mapsto g_1^{\lambda}$, \quad  $g_2\mapsto g_2^{\lambda^{-1}}$ & $\lambda< \lambda^{-1}, \in \{2,\ldots, p-1\}$& no & $\frac{p-1}{2}$\\
 \hline
 $g_1\mapsto g_1^{-1}$, \,\,  $g_2\mapsto g_1 g_2^{-1}$ &  & no & $1$\\
 \hline
  $g_1\mapsto g_2$, \quad  $g_2\mapsto  g_1^{-1} g_2^{-b}$ & $x^2+bx+1$ is an irreducible polynomial & no & $\frac{p-1}{2}$\\
 \hline
\end{tabular}
\end{table}

For non principal quandles over $G_i$, $7\leq i\leq 10$ the automorphisms of the group are determined by their image on the generators of the power-commutator presentation, which are denoted by $g_1,g_2,g_3,g_4$ with $\gamma_1(G_i)=\langle g_3,g_4\rangle$ and $Z(G)=\langle g_4\rangle$. If $i=7,8$, then $g_2,g_3,g_4$ generate a characteristic elementary abelian subgroup of rank $3$, and if $i=9,10$ then $g_3, g_4$ generate an elementary abelian subgroup of rank $2$. So the action of an automorphism is given by its restriction to the above mentioned elementary abelian subgroup (denoted by $M$) described by a matrix and by the action on the other generators.  For all the representative of the isomorphism classes of such quandles the restriction to $M$ is given by one of the following matrices
\begin{displaymath}
i=7:\quad F_{a,s}=\begin{bmatrix}
a^{-1} & 0 & 0\\
0 & 1& 0\\
 s & \frac{a-1}{2} & a
 \end{bmatrix}\qquad i=8: \quad G_{a,k}=\begin{bmatrix}
a^{-1} & 0 & 0\\
0 & 1& 0\\
 k & \frac{a-1}{2} & a
\end{bmatrix},\quad i=9,10:\quad F=\begin{bmatrix}
1 & 0\\
 0 & -1 
\end{bmatrix},
\end{displaymath}
where $0\leq a,k,s\leq p-1$ and $a\neq 0$. In Table \ref{Tab1} the isomorphism classes of such quandles are given and $w$ denotes a non quadratic residue modulo $p$.
%
%
\begin{center}
\begin{table}[!h]
\caption{Non-principal connected quandle of size $p^3$.}
\label{Tab1}
\begin{tabular}{|c|l|l|l|c|}
\hline
$\dis(Q)$ & $f|_M$ & & Parameters & \# \\
\hline
$G_7$ &  $F_{a,0}$ & $g_1\mapsto  g_1^{a}g_2^b $ & $a\in \{2,\ldots,p-2\}$, $b\in \{0,1\}$ & $2(p-3)$\\
\hline
$G_7$ &  $F_{-1,0}$ & $g_1\mapsto  g_1^{-1}g_2^b g_4^{c} $ & $(b,c)\in \{ (0,0), (1,0), (0,1)\}$& $3$\\
\hline
$G_7$ &  $F_{-1,s}$ & $g_1\mapsto  g_1^{-1}g_2^b  $ & $(s,b)\in \{ (1,0), (1,1), (w,0),(w,1)\}$& $4$\\
\hline
$G_8$ &  $G_{a,0}$ & $g_1\mapsto  g_1^{a}g_2^b  $ & $a\in \{2,\ldots,p-2\}$, $b\in \{0,\ldots p-1\}$& $p(p-3)$\\
\hline
$G_8$ &  $G_{-1,k}$ & $g_1\mapsto  g_1^{p-1}g_2^s  $ & $k\in \{0,\ldots p-1\}$, $s\in \{1,w\}$& $2p$\\
\hline
$G_8$ &  $G_{-1,k}$ & $g_1\mapsto  g_1^{p-1}  $ & $k\in \{1,w\}$&$2$\\
\hline
$G_8$ &  $G_{-1,0}$ & $g_1\mapsto  g_1^{p-1}g_4^k  $ & $k\in \{0,\ldots,p-1\}$&$p$\\
\hline
$G_9$ &  $F$ & $g_1\mapsto  g_1^{-1} g_3^k  $,\quad  $g_2\mapsto  g_2^{-1}g_3^s  $  &  $k\in \{0,1\}$, $s\in \{0,\ldots p-1\}$&$2p$\\
\hline
$ G_{10}$ &  $F$ & $g_1\mapsto  g_1^{-1} g_3^k  $,\quad  $g_2\mapsto  g_2^{-1}g_3^s  $  &  $k\in \{0,1\}$, $s\in \{0,\ldots p-1\}$&$2p$\\
\hline
\end{tabular}
\end{table}
\end{center}
\medskip

\subsection*{Bruck Loops}\label{Sec:Bruck}
It is a well known fact that Bruck Loops are in one to one correspondence with involutory latin quandles, \cite{Survey}. Moreover Bruck Loops of odd order are in one-to-one correspondence with the so-called Gamma Loops, which contains the class of commutative automorphic Loops \cite{MarkG}. The involutory latin quandles of size $p^3$ are those in Table \ref{Tab3}.
\begin{table}[!h]
\caption{Involutive latin quandles of order $p^3$.}
\label{Tab3}
\begin{tabular}{|c|l|l|}
\hline
$\dis(Q)$ & $f|_M$ &  \\
\hline
$G_7$ &  $F_{-1,0}$ & $g_1\mapsto  g_1^{-1}$  \\
\hline
$G_8$ &  $G_{-1,0}$ & $g_1\mapsto  g_1^{-1}$\\
\hline
$G_9$ &  $F$ & $g_1\mapsto  g_1^{-1}$, \quad $g_2\mapsto  g_2^{-1}$ \\
\hline
$ G_{10}$ &  $F$ & $g_1\mapsto  g_1^{-1}$, \quad $g_2\mapsto  g_2^{-1}$ \\
\hline
\end{tabular}
\end{table}
Therefore, there exist $4$ non-associative Bruck Loops or order $p^3$, so there are $7$ of them including abelian groups. 
According to the enumeration results in \cite{Petr}, there are $7$ commutative automorphic loops of size $p^3$, so we have the following:
\begin{Pro}\label{bruck loops}
There exist $7$ Bruck loops or order $p^3$ and they are in one to one correspondence with commutative automorphic loops of size $p^3$.\qed
\end{Pro} 
Propositon \ref{bruck loops} provides an answer to Problem 8,1 in \cite{Petr2} for every prime $p$ and $k=3$.

\section{Displacement group of latin $p$-quandles}\label{bound for latin}
In order to classify connected quandles of size $p^3$ we first found a bound on the order of the \disp\, group and then we analyzed the conjugacy classes of automorphisms of such groups. This could be a general strategy for the classification of other classes of connected quandles taking advantage of the minimal coset representation. We present a bound on the size of the \disp\, group for latin quandles of prime power size.

Recall that term operations of an algebra $A$ are all the (meaningful) composition of basic operations and that the subalgebra generated by a subset $X\subseteq A$ is given by
\begin{displaymath}
\setof{t\left(x_1,\ldots,x_n\right)}{\left(x_1,\ldots, x_n\right)\in X^n, \, n\in \mathbb{N},\,\text{ $t$ is a $n$-ary term operation} }.
\end{displaymath}

%
%
\begin{Lem}\label{generators}
Let $Q$ be a connected quandle, $\alpha\in Con(Q)$. If $Q/\alpha$ is generated by $[a_1], [a_2],\ldots,[a_n]$ then $Q$ is generated by $[a_1],a_2,\ldots, a_n$. In particular if $[a_1]$ is generated by $b_1,\ldots, b_m$ then $Q$ is generated by $b_1,\ldots, b_m,a_2,\ldots,a_n$.
\end{Lem}
\begin{proof}
Using Proposition 4.1 of \cite{CP} it is easy to prove that if $Q/\alpha$ is generated by $[a_1],\ldots,[a_n]$, then for every $[b]\in Q/\alpha$ there exists $h\in \langle L_{a_1},\ldots,L_{a_n}\rangle$ such that $[b]=\pi_\alpha(h)([a_1])$. So, if $c\,\alpha\,b$ then $c\in h([a_1])$ and so $c$ belongs to the subquandle generated by $[a_1],a_2,\ldots a_n$.
\end{proof}

\begin{Cor}\label{generator of p quandle}
Let $Q$ be a connected quandle of size $p^n$. The number of generators of $Q$ is less or equal to $n+1$.
\end{Cor}
\begin{proof}
If $Q$ is simple, then it is affine and it is generated by any pair of elements \cite[Theorem 3.3]{Principal}. Let $|Q|=p^n$, $\alpha\in Con(Q)$ with $|Q/\alpha|=p^k$. By induction, $Q/\alpha$, has at most $k+1$ generators and $[a]_\alpha$ has at most $n-k+1$ generators. By Lemma \ref{generators}, $Q$ has at most $k+1+n-k+1-1=n+1$ generators.
\end{proof}

\begin{The}[Upper bound for the Displacement group of latin p-quandles]
\label{bound}Let $Q$ be a latin quandle of size $p^n$, $n\geq 3$. Then $|\dis(Q)|\leq p^{\frac{n^2+n-4}{2}}$. 
\end{The}
\begin{proof}
For $n=3$ the formula holds by virtue of Proposition \ref{Lem nec cond}. If $Q$ is simple, then it is affine and $|\dis(Q)|=|Q|$. If $Q$ is not simple, let $|Q/\zeta_Q|=p^k$ with $k<n$. Then
\begin{displaymath}
|\dis(Q)|= |\dis(Q/\zeta_Q)||\dis^{\zeta_Q}|\leq |\dis(Q/\zeta_Q)||[a]_{\zeta_Q}|^{k+1},
\end{displaymath}
by virtue of Corollary \ref{generator of p quandle} and Corollary \ref{embedding of ker}. By induction, $|\dis(Q/\zeta_Q)|\leq p^{\frac{k^2+k-4}{2}}$, so we have:
\begin{equation}\label{function}
|\dis(Q)|\leq S(k)=p^{\frac{k^2+k-4}{2}}p^{(n-k)(k+1)}.
\end{equation}
The maximum of the function $S$ in \eqref{function} on integers values is in $k=n-1$, then $S(n-1)$ provides the desired bound. 
\end{proof}

\begin{example}
The bound in Theorem \ref{bound} can be improved when you specialize to a particular family of $p$-quandles: for instance the bound for nilpotent latin quandles of length $2$ and size $p^4$ can be improved from $p^8$ to $p^6$ employing a case-by-case discussion on the structure of the \disp\, group.
\end{example}

%
%


\section{Appendix}\label{appendix}
In the present Appendix we illustrate the computations that lead us to determine the conjugacy classes of the automorphisms of the relevant groups which give rise to the isomorphism classes listed in Table \ref{Tab2} and Table \ref{Tab1}. To this end, in each case we start with a presentation of the generic automorphism of the group, which can be found in \cite{tedesco}, and then proceed to impose the necessary and sufficient conditions which we have stated in Proposition \ref{Lem nec cond} and the discussion thereof.

\subsection*{Conjugacy classes of automorphisms of $\mathbb{Z}_p^2\rtimes \mathbb{Z}_p$}
In this section we will use the power-commutator presentation of the group $G=\mathbb{Z}_p^2\rtimes \mathbb{Z}_p$ as $$G=\langle g_1, g_2, g_3,\, | \, g_1^p=g_2^p=g_3^p=1, \, [g_1,g_2]=g_3\rangle.$$
The automorphisms of $G$ are given by the mappings
\begin{displaymath}
 f= \left\{
    \begin{array}{l}
g_1\mapsto g_1^a g_2^c  g_3^\alpha\\
g_2\mapsto g_1^b g_2^d g_3^\beta,\\
g_3\mapsto  g_3^{ad-bc}  
    \end{array}\right.
\end{displaymath}
where $a,b,c,d,\alpha,\beta \in \{0,\ldots,p-1\}$ and $ad-bc\neq 0$. We will denote such automorphism as $f=f(a,b,c,d,\alpha,\beta)$.
%
%
%

\begin{Pro}\label{latin principal}
A set of representatives of conjugacy classes of $\aut{G}$ with no fixed points is: \begin{itemize}
\item[(i)] $f_{\lambda, \mu} = f(\lambda, 0,0,\mu,0,0)$, where $\lambda \leq \mu \in \{2,\ldots, p-1\}, \lambda \mu \neq 1$.
\item[(ii)] $h_{\lambda} = f(\lambda,0,1,\lambda,0,0)$,  where $\lambda \in \{2,\ldots, p-2\}$.
\item[(iii)] $f_{p} = f(0, 1,-a,-b,0,0)$, where $p=x^2+bx+a$ is an irreducible polynomial and  $a\neq 1$.
\end{itemize}
The number of such conjugacy classes is $p^2-2p-1$.
\end{Pro}
\begin{proof}
Let $f\in \aut{G}$ with no fixed points and let $M=f_{\gamma_1(G)}$ be the matrix of its action on the quotient by $\gamma_1(G)$. Since $f(g_3)=g_3^{ab-cd}$, then $M$ has determinant different from $1$. We can have the following cases:
\begin{itemize}
\item[(i)] $M$ is diagonalizable with eigenvalues $\lambda\leq \mu$ and $\lambda,\mu,\lambda \mu \neq 1$. So there exist $x,y,z,t$ such that $h=h(x,y,z,t,0,0)$ satisfies: 
$$h f h^{-1}=f(\lambda, 0,0,\mu,\delta,\rho)$$
 for some $\delta,\rho$. Conjugating $f(\lambda, 0,0,\mu,\delta,\rho)$ by $w=w(1,0,0,1,x,y)$ with $x=\frac{\delta}{\lambda(1-\mu)}$ and $y=\frac{\rho}{\mu(1-\lambda)}$ we obtain $f_{\lambda,\mu}$. Finally, note that $f_{\lambda,\mu}$ and $f_{\mu,\lambda}$ are conjugate by $w=w(0,1,1,0,0,0)$. 

\item[(ii) ]  $M$ has one eigenvalue $\lambda\neq 1$ of multiplicity $2$ and it is not diagonalizable. So there exist $x,y,z,t$ such that $h=h(x,y,z,t,0,0)$ satisfies: 
$$h f h^{-1}=f(\lambda, 0,1,\lambda,\delta,\rho)$$
 for some $\delta,\rho$. Conjugating $f(\lambda,0,1,\lambda,\delta,\rho)$ by $w=w(1,0,0,1,x,y)$ with $x=\frac{\delta}{\lambda(1-\lambda)}$ and $y=\frac{\rho-x}{\lambda(1-\lambda)}$ we obtain $h_{\lambda}$.

\item[(iii)] $M$ has no eigenvalues. So there exists $x,y,z,t$ such that $h=h(x,y,z,t,0,0)$ satisfies 
$$h f h^{-1}=f(0, 1,-a,-b,\delta,\rho)$$ 
for some $\delta,\rho$ and $a\neq 1$ such that $p=x^{2}+bx+a$ is an irreducible polynomial. Conjugating $f(0, 1,-a,-b,\delta,\rho)$ by $w=w(1,0,0,1,x,y)$ with $x=\frac{a\delta-\rho-b\delta}{a^2-ab+a}=\frac{a\delta-\rho-b\delta}{p(-a)}$ (since $p$ is irreducible $x$ is well defined) and $y=\delta-ax$ we obtain $f_p$.
\end{itemize} 
The automorphisms in the list are not conjugate since the induced automorphisms on $G/\gamma_1(G)$ are not.  There are $\frac{(p-1)(p-3)}{2}$ different $f_{\lambda,\mu}$, $p-3$ different $h_\lambda$ and $\frac{(p-1)^2}{2}$ irreducible polynomial with constant term different from $1$ (since there are $\frac{p-1}{2}$ irreducible polynomials of degree two and constant term equal to $1$ \cite[Theorem 3.5]{pol}). So the number of conjugacy classes is
\begin{displaymath}
\frac{(p-1)(p-3)}{2}+\frac{(p-1)^2}{2}+p-3=(p-1)(p-2)+p-3=p^2-2p-1.
\end{displaymath}

\end{proof}

\begin{Pro}
A set of representatives of conjugacy classes of automorphisms $f$ such that $Fix(f)=Z(G)$ and $Fix(f_{\gamma_1(G)})=1$ is:
\begin{itemize}
\item[(i)] $f_{\lambda} = f(\lambda, 0,0,0,\lambda^{-1},0,0)$, where $ \lambda< \lambda^{-1}, \in \{2,\ldots, p-1\}$.
\item[(ii)] $h_{-1} = f(-1, 1,0,0,-1,0,0)$.
\item[(iii)] $f_{p,1} = f(0, 1,0,-1,-b,0,0)$, where $p= x^2+bx+1$ is an irreducible polynomial.
\end{itemize}
The number of such conjugacy classes is $p$.
%
\end{Pro}
\begin{proof}
The proof is analogous to the proof of Proposition \ref{latin principal}, taking into account that we are now dealing with automorphisms with $Fix(f)=Z(G)$, and so $f_{\gamma_1(G)}$ has determinant $1$.

There are $\frac{p-1}{2}$ different $f_\lambda$ and $\frac{p-1}{2}$ irreducible polynomials of degree two and constant term equal to $1$ \cite[Theorem 3.5]{pol}. Hence there are $p$ conjugacy classes.
\end{proof}
%
%

\subsection*{Conjugacy classes of automorphisms of the relevant groups of order $p^4$ for $p>3$}
\subsubsection*{Case $G=G_7$}
In this section we will use the power-commutator presentation of the group $G_7$ as
\begin{displaymath}
G_7=\langle g_1, g_2, g_3, g_4, \, | \, [g_2,g_1]=g_3, \, [g_3,g_1]=g_4\rangle.
\end{displaymath}
Using an inductive argument and that $g_1^{-a}g_2g_1^a= g_2 g_3^a g_4^{\frac{a(a-1)}{2}}$ and $g_1^{-a}g_3 g_1^a=g_3 g_4^a$, we can prove the following Lemma:
\begin{Lem}\label{normalized reps}
Let $n\in \mathbb{Z}$ $a,b,c\in\{ 0,\ldots , p-1\}$ and $g_1,g_2,g_3\in G$. Then
\begin{equation}
(g_1^a g_2^b g_3^c)^n=g_1^{an} g_2^{nb} g_3^{nc+ab\binom{n}{2}}g_4^{ac\binom{n}{2}+b\rho(a,n)},
\end{equation}
where $\rho(a,n)=\frac{a}{2}\binom{n}{2}(\frac{a(2n-1)}{3}-1)$.\qed
\end{Lem}

Following \cite{tedesco} and since $g_2,g_3,g_4$ generate an elementary abelian subgroup of rank $3$, which we denote here by $M$, the automorphisms of $G$ are given by the following collection of data:
\begin{small}
 
\begin{displaymath}
h= \left\{
    \begin{array}{l}
H=h|_M=\begin{bmatrix}
x & 0 & 0\\
y & tx& 0\\
z & ty+xt\frac{t-1}{2} & t^2x
\end{bmatrix} ,  \qquad 
g_1\mapsto  g_1^t g_2^{u} g_3^{\alpha} g_4^{\beta}\end{array}\right.,
\end{displaymath}
 \end{small}
where $u,y,z,\alpha,\beta\in \{0,\ldots, p-1\}$ and $x,t\in \{1,\ldots,p-1\}$. A direct computation tells us that $|\aut{G}|=(p-1)^2p^5$. We denote $h$ by $h(t,u,\alpha,\beta,x,y,z)$, and it is straightforward to verify that $h$ satifies \eqref{condition_on_aut} if and only if $t x=1$ and $t\neq 1$.

\begin{Lem}\label{normalized rep}
Let $f=f(a,b,\alpha,\beta,a^{-1},c,d)\in \aut{G}$ satisfying condition \eqref{condition_on_aut} and let $f_{M}$ be the action of $f$ on the quotient by $M$. Then: 
\begin{itemize}
\item[(i)] $f_{M}=a$ is invariant under conjugation. 
\item[(ii)] If $
a\neq 1, p-1$. Then $f$ is conjugate to $f_0=f_0(a,b,\gamma,\delta,a^{-1}0,0)$, i.e.
\begin{small}
\begin{displaymath}
f_0= \left\{
    \begin{array}{l}
 
F_0=\begin{bmatrix}
a^{-1} & 0 & 0\\
0 & 1& 0\\
 0& \frac{a-1}{2} & a
\end{bmatrix} ,\qquad g_1\mapsto  g_1^a g_2^{b} g_3^{\gamma} g_4^{\delta}  \end{array}.\right.
\end{displaymath} 
\end{small}
for some $\gamma,\delta\in \{0,\ldots, p-1\}$.
\end{itemize}

\end{Lem}
\begin{proof}
(i) Let $f$ and $g$ be conjugated. Then $f_{M}$ and $g_{M}$ are conjugate in $\aut{G/M}$. Since $G/M\cong \mathbb{Z}_p$, then $f_M=g_M$. 

(ii) Conjugating $f$ by $h=h(1,0,1,0,0,1,y,z)$ where
\begin{displaymath}
y=\frac{ac}{a-1},\qquad z=\frac{2a^2 c^2-a^3c+a^2c+2a^2d-2ad}{2(a-1)(a^2-1)},
\end{displaymath}
we obtain $f_0$.
\end{proof}

Let $a\neq  1,p-1$. By Lemma \ref{normalized rep}, up to conjugation any automorphism of $G$ satisfying \eqref{condition_on_aut} is completely determined by the image of $g_1$.
\begin{Pro}
A set of representative of conjugacy classes of automorphisms satisfying \eqref{condition_on_aut} with $a\neq 1, p-1$ is given by 
\begin{eqnarray*}
f_a &:& g_1\mapsto g_1^a, \\
 g_a &:& g_1\mapsto g_1^{a}g_2.
\end{eqnarray*}
In particular, there are $2(p-3)$ such conjugacy classes.
\end{Pro}

\begin{proof}
We first compute the centralizers of $f_a$ and $g_a$. Let $h=h(t,u,\alpha,\beta,x,y,z)\in C_{\aut{G}}(f_a)$. Then $H$ centralizes $F_a$ and so $y=z=0$. 

If $h=h(t,u,\alpha,\beta,x,0,0)\in C_{\aut{G}}(f_a)$, then $h_{\gamma_1(G)}$ centralizes ${f_a}_{\gamma_1(G)}$ and so $u=0$. Then $h=h(t,0,\alpha,\beta,x,0,0)\in C_{\aut{G}}(f_a)$ if and only if $\alpha=0$.
Accordingly $|C_{\aut{G}}(f_a)|=p(p-1)^2$, since $t,x$ and $\beta$ are free.

We have that $h=h(t,u,\alpha,\beta,x,0,0)\in C_{\aut{G}}(g_a)$, if and only if $t,u,\alpha$ and $x$ are solutions of the following system of linear equations over $\mathbb{Z}_p$:
\begin{displaymath}
\left\{\begin{array}{l} u(a^{2}-1)=a(t-x) \\
\alpha(a-1)=a\binom{t}{2}-tu\binom{a}{2}\\
\alpha(a-1)(1-ta)=2u\rho(t,a)-2\rho(a,t)
   \end{array}\right..
\end{displaymath}

A straightforward calculation shows that $tx=1$ and $u,\alpha$ are uniquely determined by $t$ and $x$. Therefore $|C_{\aut{G}}(g_a)|=p(p-1)$.

Then, $f_a$ and $g_a$ are not conjugate since the sizes of their centralizers are different. The sizes of conjugacy classes of $f_a$ and $g_a$ are respectively $p^4$ and $(p-1)p^4$. The automorphisms of $G$ satisfying \eqref{condition_on_aut} with fixed $a$ are $p^5=p^4+(p-1)p^4$, so $f_a$ and $g_a$ are the representatives of their conjugacy classes. 
\end{proof}

We now discuss the case $a=-1$. 
\begin{Lem}\label{standard rep i=7 a=-1}
Let $f=f(-1,b,\alpha,\beta,-1,c,d)\in \aut{G}$ satisfying condition \eqref{condition_on_aut}. Then: 
\begin{itemize}
\item[(i)]
$f$ is conjugate to 
\begin{small}
\begin{displaymath}
f_0= \left\{
    \begin{array}{l}
F_0=\begin{bmatrix}
-1 & 0 & 0\\
0 & 1& 0\\
 \frac{c(c+1)+2d}{2} & -1 & -1
\end{bmatrix} ,\qquad g_1\mapsto  g_1^{-1} g_2^{b} g_3^{\gamma} g_4^{\delta}  \end{array}.\right.
\end{displaymath} 
\end{small}
for some $\gamma,\delta,\in \{0,\ldots,p-1\}$.  
\item[(ii)] If $g(-1,b,\alpha,\beta,-1,0,s)$ and $h(-1,u,\gamma,\delta,-1,0,s^\prime)$ are conjugate, then $s=s^{\prime}r^2$ for some $r\neq 0$.
\end{itemize}
\end{Lem}
\begin{proof}
(i) Conjugating $f$ by $h=h(1,0,0,0,1,\frac{c}{2},0)$ you obtain $f_0$.

(ii) The second statement follows by direct computations (it is enough to consider the restriction of $h$ and $g$ to $M$).
\end{proof}

According to Lemma \ref{standard rep i=7 a=-1}, every automorphism is conjugate to one of the standard form $g(-1,b,\alpha,\beta,-1,0,s)$. It follows from item (ii) that automorphisms with $s=0$ and $s^\prime\neq 0$ are not conjugate. The following Lemma provides a list of representatives of conjugacy classes in the case $s=0$.
. 
\begin{Pro}
A set of representatives of the conjugacy classes of elements of the form $f=f(-1,b,\alpha,\beta,-1,c,d)\in \aut{G_7}$ satisfying \eqref{condition_on_aut}
with $c(c+1)+2d=0$ is:
\begin{itemize}
\item[(i)]  $\phi=\phi(-1,0,0,0,-1,0,0)$ (whenever $f^2=1$).
\item[(ii)]  $\varepsilon=\varepsilon(-1,0,0,1,-1,0,0)$ (whenever $f_{\gamma_1(G)}=-I$ and $f^2\neq 1$).
\item[(iii)] $\sigma=\sigma(-1,1,0,0,-1,0,0)$ (whenever $f_{\gamma_1(G)}\neq -I$).
\end{itemize}
\end{Pro}
\begin{proof}
The automorphisms $\phi$ and $\varepsilon$ are not conjugate since they have different orders
and $\sigma$ is not conjugate to $\phi$  nor to $\varepsilon$ since ${\sigma}_{\gamma_1(G)}$ is not conjugate to ${\phi}_{\gamma_1(G)}$ nor to ${\varepsilon}_{\gamma_1(G)}$. According to Lemma \ref{standard rep i=7 a=-1}(i), we can assume that $f=f(-1,b,\alpha,\beta,-1,0,0)\in \aut{G}$.

(i) Let $f=f(-1,b,\alpha,\beta,-1,0,0) \in \aut{G}$ satisfying \eqref{condition_on_aut}. It is easy to verify that $f^2=1$ if and only if $b=0$ and $\alpha+\beta=0$ i.e. 
we can assume that $f=f(-1,0,\alpha,-\alpha,0,0)$. The automorphisms $h=h(1,0,-\frac{\alpha}{2},0,1,0,0)$ fulfils $\phi=h^{-1}fh$. 
%
%

(ii) Since $f$ is not involutory, $\alpha+\beta\neq 0$. Let $h=(1,-\alpha,0,0,\alpha+\beta,0,0)\in \aut{G}$ then $\varepsilon=h^{-1} f h$.


(iii) Let $h=h(1,0,-\frac{\alpha}{2},\frac{\alpha}{2},b,0,\beta)$, then $\sigma=h^{-1} f h$.
\end{proof}

Now we consider the conjugation classes of automorphisms of the form $f=f(-1,b,\alpha,\beta,-1,0,s)$ with $s\neq 0$.

%

\begin{Pro}
Let $w$ be a non quadratic residue modulo p. A set of representatives of conjugacy classes of elements of the form $f=f(-1,b,\alpha,\beta,-1,c,d)\in \aut{G_7}$ satisfying \eqref{condition_on_aut} with $\theta=c(c+1)+2d\neq 0$ is: 
\begin{itemize}
\item[(i)]  $f_1=f_1(-1,0,0,0,-1,0,1)$, (whenever $f_{M}=-I$ and $\theta$ is a quadratic residue module $p$).
\item[(ii)]  $f_w=f_w(-1,0,0,0,-1,0,w)$, (whenever $f_{M}=-I$  and $\theta$ is not a quadratic residue module $p$).
\item[(iii)] $h_{1}=h_{1}(-1,1,0,0,-1,0,1)$, (whenever $f_{M}\neq -I$ and $\theta$ is a quadratic residue module $p$).
\item[(iv)] $h_{w}=h_{w}(-1,1,0,0,-1,0,w)$, (whenever $f_{M}\neq -I$ and $\theta$ is not a quadratic residue module $p$).
\end{itemize}
\end{Pro}

\begin{proof}
The automorphisms $f_s$ and $h_s$, for $s=1,w$ are not conjugate since ${f_s}_M$ and ${h_s}_M$ are not conjugate. Moreover, $f_1$ and $f_w$ (resp. $h_1$ and $h_{w}$) are not conjugate, since ${f_{1}}|_{M}$ and ${f_{w}}|_{M}$ (resp. ${h_{1}}|_{M}$ and ${h_{w}}|_{M}$) are not conjugate by Lemma \ref{invariants}(ii).

Let $f_s=f_s(-1,0,0,0,-1,0,s)$ and $h_s=h_s(-1,1,0,0,-1,0,s)$. Every $f=f(-1,b,\alpha,\beta,-1,0,s)$ is conjugate to $f_s$ or to $h_s$ for some $s\neq 0$ by $h=h(t,u,\gamma,\delta,x,y,z)$ where $u=\beta s^{-1}$ and $\gamma	=\frac{\alpha+u}{2}$ i.e. $h f h^{-1}=f_s$ and $h f h^{-1}=h_s$. \\
On the other hand, setting $y=z=0$, $x=t=r$ and $u=\gamma=\delta=0$ we have that $h f_s h^{-1}=f_{r^2s}$. Setting $t=x=r$, $y=z=\beta=0$, $u=\frac{1-r}{2}$ and $\alpha=sr^2u-\rho(-1,r)-r+\frac{ru(r+1)}{2}$ we have $h h_s h^{-1}=h_{r^2s}$. In particular, every $f_s$ (resp. $h_s$) is conjugate to $f_1$ (resp. $h_1$) if $s$ is a quadratic residue modulo $p$ and  $f_s$ (resp. $h_s$) is conjugate to $f_w$ (resp. $h_w$) otherwise. 
%
%
%
\end{proof}

\subsubsection*{Case $G=G_8$}
In this section we will use the power-commutator presentation of the group $G_8$ as
\begin{displaymath}
G_8=\langle g_1, g_2, g_3, g_4, \, | \, [g_2,g_1]=g_3, \, [g_3,g_1]=g_1^p=g_4\rangle.
\end{displaymath}
According to \cite{tedesco} and since $g_2,g_3,g_4$ generate an elementary abelian subgroup of rank $3$ denoted by $M$, the automorphisms of $G_8$ are given by:
\begin{small}
\begin{displaymath}
h= \left\{
    \begin{array}{l}
H=h|_M=\begin{bmatrix}
t^{-1} & 0 & 0\\
y & 1& 0\\
z & ty+\frac{a-1}{2} & t
\end{bmatrix} ,\qquad g_1\mapsto  g_1^t g_2^{u} g_3^{\alpha} g_4^{\beta}  \end{array}.\right.
\end{displaymath} 
\end{small}
where $u,y,z.\alpha,\beta\in \{0,\ldots, p-1\}$ and $t\in \{1,\ldots,p-1\}$. We have that $|\aut{G}|=(p-1)p^5$. We denote $h$ by $h(t,u,\alpha,\beta,y,z)$, and it is straightforward to verify that $h$ satifies \eqref{condition_on_aut} if and only if $t \neq 1$. Note that for every fixed $a$ there exists $p^5$ automorphisms satisfying \eqref{condition_on_aut}.
\begin{Lem}
Let $f=f(a,b,\alpha,\beta,c,d)\in \aut{G_8}$ satisfying condition \eqref{condition_on_aut} . Then
\begin{itemize}
\item[(i)] $f_{M}=a$ is invariant under conjugation. 
\item[(ii)] If $a\neq  1$, $f$ is conjugate to $f_0=f_0(a,b,\gamma,\delta,0,0)$, i.e.
\begin{small}
\begin{displaymath}
f_0= \left\{
    \begin{array}{l}
F_0=\begin{bmatrix}
a^{-1} & 0 & 0\\
0 & 1& 0\\
 g(a,c,d) & \frac{a-1}{2} & a
\end{bmatrix} , \qquad g_1\mapsto  g_1^a g_2^{b} g_3^{\gamma} g_4^{\delta}
  \end{array},\right.
\end{displaymath}
\end{small}
for some $\gamma,\delta$, where $g(a,c,d)$ is a polynomial function of $a,c$ and $d$ such that $g(a,0,0)=0$. 
\end{itemize}
\end{Lem}
\begin{proof}
(i) The same argument of \ref{standard rep i=7 a=-1}(i) applies.

(ii) Conjugating $f$ by $h=h(1,0,1,0,y,z)$ where $ y=\frac{ac}{(a-1)}$ and $z=0$ we obtain $f_0$.
%
\end{proof}

\begin{Pro}
Let $a\neq 1,p-1 $. A set of representatives of conjugacy classes of automorphisms of $G_8$ satisfying \eqref{condition_on_aut} is:
\begin{small}
\begin{displaymath}
f_{a,b}= \left\{
    \begin{array}{l}
F_a=\begin{bmatrix}
a^{-1} & 0 & 0\\
0 & 1& 0\\
 0 & \frac{a-1}{2} & a
\end{bmatrix} ,  \quad g_1\mapsto  g_1^{a}g_4^b ,
 \end{array},   \right.
\end{displaymath}
\end{small}
for $b\in \{0,\ldots p-1\}$. In particular, there are $p(p-3)$ such conjugacy classes.
\end{Pro}
\begin{proof}
We compute the centralizer of $f_{a,b}$. We have that $h=h(t,u,\alpha,\beta,y,z)\in C_{\aut{G_8}}(f_{a,b})$ if and only if $\alpha=u=y=z=0$. Hence $|C_{\aut{G_8}}(f_{a,b})|=p(p-1)$. 
It is easy to see that if $f_{a,b}$ and $f_{a,c}$ are conjugate, then $b=c$. Since for every $a$ there are $p^5$ automorphisms of $G$ satisfying \eqref{condition_on_aut} and 
\begin{displaymath}
\sum_{b\in \{ 0,\ldots,p-1\}}[\aut{G}:C(f_{a,b})]=p\frac{(p-1)p^5}{p(p-1)}=p^5,
\end{displaymath}
we are done.
\end{proof}

Now we discuss the case $a=p-1$. 

\begin{Lem}\label{invariants}
Let $f=f(p-1,b,\alpha,\beta,c,d)\in \aut{G}$ satisfying condition \eqref{condition_on_aut}. Then
\begin{itemize}
\item[(i)] $f$ is conjugate to $f_0=f_0(p-1,b,\gamma,\delta,0,\frac{c^2+2d}{2})$, i.e.
\begin{small}
\begin{displaymath}
f_0= \left\{
    \begin{array}{l}
F_0=\begin{bmatrix}
-1 & 0 & 0\\
0 & 1& 0\\
 \frac{c(c+1)+2d}{2} & -1 & -1
\end{bmatrix} ,  \qquad g_1\mapsto  g_1^{-1} g_2^{b} g_3^{\gamma} g_4^{\delta}
 \end{array},\right.
\end{displaymath}  
\end{small}
for some $\gamma,\delta\in \{0,\ldots,p-1\}$.
\item[(ii)] If $g(p-1,b,\alpha,\beta,0,s)$ and $h(p-1,b^\prime,\gamma,\delta,0,s^\prime)$ are conjugate, then $s^{\prime}=sr^2$ and $b^\prime =\frac{b}{r^2}$ for some $r\neq 0$. Therefore $bs=b^\prime s^\prime$.
\end{itemize} 
\end{Lem}
\begin{proof}
(i) Conjugating $f=f(p-1,b,\alpha,\beta,c,d)$ by $h=h(1,0,0,0,\frac{c}{2},0)$ we obtain $f_0$.

(ii) Let $g(p-1,b,\alpha,\beta,0,s)$ and $h(p-1,b^\prime,\gamma,\delta,0,s^\prime)$ be conjugate. Since $g|_M$ is conjugate to $h|_M$, then there exists $r\neq 0$ such that $s^\prime=s r^2$. Since $g_{\gamma_1(G)}$ and $h_{\gamma_1(G)}$ are conjugate then $b^\prime =\frac{b}{r^2}$.
\end{proof}

Now we consider the classes of automorphisms $f$ such that $f_{\gamma_1(G)}\neq -I$.
\begin{Pro}\label{centralizers i=8 SI}
Let $a=p-1$ and $w$ be a non quadratic residue modulo $p$. A set of representatives of conjugacy classes of automorphisms of $G_8$ satisfying \eqref{condition_on_aut} and such that $f_{\gamma_1(G)}\neq -I$ is:
\begin{small}
\begin{displaymath}
f_{k,s}= \left\{
    \begin{array}{l}
F_{k,s}=\begin{bmatrix}
-1 & 0 & 0\\
0 & 1& 0\\
 k & -1 & -1
\end{bmatrix} ,  \qquad g_1\mapsto  g_1^{p-1} g_2^{s},  
 \end{array}\right.
\end{displaymath}
\end{small}
for $k\in \{0,\ldots p-1\}, s\in \{1,w\}$.  In particular, there are $2p$ such conjugacy classes.
\end{Pro}
\begin{proof}
Clearly the $f_{k,s}$  are not conjugate by Lemma \ref{invariants} (ii). Let $h=h(t,u,\gamma,\delta,y,z)\in \aut{G}$. Then $h\in C_{\aut{G}}(f_{k,s})$ if and only if 
\begin{align*}
t=1, \quad y=0, \quad z=-2s,\quad \alpha=\frac{u+s}{2},\quad 
\text{or}\quad t=p-1, \quad y=0, \quad z=0,\quad \alpha=\frac{s-u}{2}.
 \end{align*}
 In both cases $u$ and $\beta$ are free. So $|C_{\aut{G_8}}|=2p^2$. Fix $k\in \{0,\ldots, p-1\}$. According to Lemma \ref{invariants}(i), the automorphisms conjugate to $f_{k,1}$ or to $f_{k,w}$ are those for which:
\begin{displaymath}
b \frac{c^2+2d}{2}=k.
\end{displaymath}
So for a given $k$ there are $p^3(p-1)$ many of them. Therefore,
\begin{displaymath}
\sum_{k \in \{0,\ldots p-1 \}} \sum_{s\in \{1,w\}} [\aut{G}:C_{k,s}]=p\left(\frac{p^3(p-1)}{2}+\frac{p^3(p-1)}{2}\right)=p^4(p-1),
\end{displaymath}
 and so $\setof{f_{s,k}}{k\in \{0,\ldots,p-1\}\, , s\in \{1,w\}}$ is a set of representatives.
%

\end{proof}

\begin{Lem}\label{centralizers i=8 not SI}
Let $a=p-1$ and $w$ be a non quadratic residue modulo $p$. A set of representatives of conjugacy classes of automorphisms of $G_8$ satisfying \eqref{condition_on_aut} and such that $f_{\gamma_1(G)}= -I$ and $\frac{c^2+2d}{2}\neq 0$ is:
\begin{small}
\begin{displaymath}
f_{s}= \left\{
    \begin{array}{l}
F_{s}=\begin{bmatrix}
-1 & 0 & 0\\
0 & 1& 0\\
 s & -1 & -1
\end{bmatrix} ,  \qquad g_1\mapsto  g_1^{p-1} ,  
 \end{array}\right.
\end{displaymath}
\end{small}
for $s\in \{1,w\}$.
\end{Lem}
\begin{proof}
Let $s\neq 0$ and $f=f(p-1,0,\alpha,\beta,0,s)$, $f_s=f(p-1,0,0,0,0,s)$. Conjugating $f$ by $h=h( 1,u,\gamma,0,0,0)$ where $u=\frac{\alpha-\beta}{2}$ and $\gamma=\frac{(1-s)\alpha-\beta}{2s}$ we obtain $f_s$. 
%
 Let $r\neq 0$ and take $h=h(r,0,0,0,0,0)$ then $f_s=h^{-1}f_{sr^2}h$. So, $\setof{f_s}{s\in \{1,w\}}$ is a set of representatives.
\end{proof}
%
%

\begin{Pro}
A set of representatives of conjugacy classes of automorphisms of $G_8$ satisfying \eqref{condition_on_aut} and such that $f_{\gamma_1(G)}= -I$ and $\frac{c^2+2d}{2}= 0$ is:
\begin{small}
\begin{displaymath}
f_{\beta}= \left\{
    \begin{array}{l}
F=\begin{bmatrix}
-1 & 0 & 0\\
0 & 1& 0\\
 0 & -1 & -1
\end{bmatrix} ,  \qquad g_1\mapsto  g_1^{p-1}g_4^\beta , 
 \end{array}\right.
\end{displaymath}
\end{small}
for $ \beta\in \{0,\ldots,p-1\}$. In particular, there are $p$ such conjugacy classes.
\end{Pro}
\begin{proof}
Let $C_\beta=C_{\aut{G_8}}(f_\beta)$. Then $h=h(r,u,\alpha,\beta,y,z)\in C_\beta$ if and only if $y=0$ and $\alpha=\frac{ru}{2}$. Hence $|C_\beta|=(p-1)p^3$ and then $[\aut{G_8}:C_\beta]=p^2$. Moreover if $f_\beta$ and $f_\gamma$ are conjugate then $\beta=\gamma$. Since there are $p^3$ automorphisms of $G_8$ satisfying the hypothesis then $\setof{f_\beta}{\beta\in \{0,\ldots,p-1\}}$ is a set of representatives.
%
%
%
\end{proof}

\subsubsection*{Case $G=G_9, G_{10}$}

In this section we will use the power-commutator presentation of the group $G_i$, for $i=9,10$ as
\begin{displaymath}
G_i=\langle g_1, g_2, g_3, g_4, \, | \, [g_2,g_1]=g_3, \, [g_3,g_1]=g_2^p=g_4^w\rangle,
\end{displaymath}
where $w$ is either $1$ or a primitive root of unity.

According to \cite[Section 5.7, 5.8]{tedesco} the automorphisms of $G_i$ have the same syntactic description and since $g_3,g_4$ generate an elementary abelian subgroup of rank $2$ they can be represented as
\begin{small}
\begin{displaymath}
h_{\pm}= \left\{
    \begin{array}{l}
g_1\mapsto  g_1^{\pm 1}  g_3^{\alpha} g_4^{\beta}, \qquad
H=\begin{bmatrix}
\pm t & 0\\
\pm u & t 
\end{bmatrix}\\
g_2\mapsto  g_2^{t}  g_3^{u} g_4^{\gamma} ,  \end{array}.\right.
\end{displaymath} 
\end{small} and denoted by $h_{\pm}(\alpha,\beta,t,u,\gamma)$ where $u,\alpha,\beta,\gamma\in \{0,\ldots, p-1\}$ and $t\in \{1,\ldots,p-1\}$. Then $|\aut{G}|=2(p-1)p^4$. It is straightforward to verify that $h_{\pm}$ satisfies \eqref{condition_on_aut} if and only 
\begin{small}
\begin{displaymath}
h_{-}= \left\{
    \begin{array}{l}
g_1\mapsto  g_1^{- 1}  g_3^{\alpha} g_4^{\beta},\qquad H=\begin{bmatrix}
 1 & 0\\
 - u & -1 
\end{bmatrix} ,\\
g_2\mapsto  g_2^{-1}  g_3^{u} g_4^{\gamma}  \end{array}.\right.
\end{displaymath} 
\end{small}
and so there are $p^4$ many of them.
\begin{Lem}\label{representatives}
Let $f\in \aut{G_i}$, for $i=9,10$, satisfying condition \eqref{condition_on_aut}. Then: 
\begin{itemize}
\item[(i)] $f$ is conjugate to $f_0=f_0(\alpha,\beta,-1,0,\gamma)$, i.e.
\begin{small}
\begin{displaymath}
f_0= \left\{
    \begin{array}{l}
g_1\mapsto  g_1^{-1} g_3^{\alpha} g_4^{\beta} ,\qquad F=\begin{bmatrix}
1 & 0 \\
0 & -1
\end{bmatrix} ,\\
g_2 \mapsto g_2^{-1} g_4^\gamma   \end{array}.\right.
\end{displaymath} 
\end{small}
for some $\alpha,\beta, \gamma\in \{0,\ldots,p-1\}$.  
\item[(ii)] If $g(\alpha,\beta,-1,0,\gamma)$ and $h(\epsilon,\rho,-1,0,\delta)$ are conjugate, then $\gamma=\delta$.
\end{itemize}
\end{Lem}

\begin{proof}
(i) Conjugating $f=f(\alpha,\beta,-1,u,\gamma)$ by $h=h_+(0,0,1,\frac{u}{2},0)$ we obtain $f_0$. 

(ii) Let $g=g(\alpha,\beta,-1,0,\gamma)$ and $h=h(\epsilon,\rho,-1,0,\delta)$ be conjugate. By looking at the image of $g_2$ of $g$ and $h$, we have $\delta=\gamma$.
%


\end{proof}

%
%

\begin{Pro}
A set of representatives of conjugacy classes of automorphisms of $G_i$, $i=9,10$, satisfying \eqref{condition_on_aut} is:
\begin{small}
\begin{displaymath}
f_{k, \gamma}= \left\{
    \begin{array}{l}
g_1\mapsto  g_1^{- 1}  g_3^{k}, \qquad 
F=\begin{bmatrix}
1 & 0\\
 0 & -1 
\end{bmatrix} ,  \\
g_2\mapsto  g_2^{-1}   g_4^{\gamma} \end{array}.\right.
\end{displaymath} 
\end{small}
for $\gamma\in \{0,\ldots p-1\}$ and $k=0,1$. In particular there are $2p$ such conjugacy classes.
\end{Pro}
\begin{proof}
Let $h_{\pm}=h_{\pm}(\alpha,\beta,x,y,z)$ and $C_{k,\gamma}= C_{\aut{G_i}}(f_{k, \gamma})$. Then
\begin{eqnarray*}
h_{+}\in C_{k,\gamma} &\Leftrightarrow & y=0 ,\quad \alpha=-\frac{k}{2},\quad k(1-2x)=0\\
h_{-}\in C_{k,\gamma} &\Leftrightarrow & y=0 ,\quad \alpha=k, \qquad k(x+1)=0.
 \end{eqnarray*}
 So the size of $C_{k, \gamma}$ is $
2p^2(p-1)$ if $k =0$ and $2p^2$ otherwise.
%
%
%
In particular $f_{k,\gamma}$ and $f_{0,\gamma}$ are not conjugate since their centralizers have different sizes. There are $p^4$ automorphisms satisfying \eqref{condition_on_aut}. By virtue of Lemma \ref{representatives}(ii), $f_{\gamma,k}$ and $f_{\gamma,l}$ are not conjugate if $k\neq l$. Therefore,
\begin{displaymath}
\sum_{\gamma\in \{0,\ldots p-1 \}} \sum_{k=0}^1 [\aut{G}:C_{k,\gamma}]=p (p^2+p^2(p-1))=p^4,
\end{displaymath}
 and the statement follows.
\end{proof}

%
%

\bibliographystyle{amsalpha}
\bibliography{references} 

\end{document}